\documentclass[12pt, leqno]{article}
\usepackage{amsmath}
\usepackage{amsthm}
\usepackage{amssymb}
\usepackage{latexsym}
\usepackage{graphicx}
\allowdisplaybreaks[1]
\usepackage{amsfonts}
\usepackage{ascmac}
\usepackage{color}
\usepackage{enumerate}

\theoremstyle{definition}
\theoremstyle{plain}
\allowdisplaybreaks[1]
\date{}
\newtheorem{Thm}{Theorem}[section]
\newtheorem{Prop}[Thm]{Proposition}
\newtheorem{Lemma}[Thm]{Lemma}

\newcommand{\p}{\partial}

\newcommand{\dis}{\displaystyle}

\newcommand{\norm}{\parallel}

\newcommand{\Q}{{\mathbb Q}}

\newcommand{\N}{{\mathbb N}}
\newcommand{\R}{{\mathbb R}}

\newcommand{\ep}{\varepsilon }
\newcommand{\2}{\frac{1}{2} }
\newcommand{\wto}{\rightharpoonup}

\newcommand{\weakto}{\rightharpoonup}

\def\text#1{\mbox{#1 }}

\title{\bf Existence of global weak solutions of inhomogeneous incompressible Navier-Stokes equations with mass diffusion}
\author{Eliott Kacedan
\footnote{Department of Mathematics, Faculty of Science and Technology, Keio University, 3-14-1 Hiyoshi, Kohoku-ku, Yokohama, 223-8522, Japan. E-mail:  eliott.kacedan@keio.jp}
  \, and  Kohei Soga
\footnote{Department of Mathematics, Faculty of Science and Technology, Keio University, 3-14-1 Hiyoshi, Kohoku-ku, Yokohama, 223-8522, Japan. E-mail:  soga@math.keio.ac.jp 
}}
\begin{document}
\maketitle
\begin{abstract} 
\noindent  This paper proves existence of a global weak solution  to the inhomogeneous (i.e., non-constant density) incompressible Navier-Stokes system with mass diffusion. The system is well-known as  the Kazhikhov-Smagulov model. The major novelty of the paper is to deal with the Kazhikhov-Smagulov model possessing the non-constant viscosity without any simplification of higher order nonlinearity.  Any global weak solution is shown to have a long time behavior that is consistent with  mixing phenomena of miscible fluids. The results also contain a new compactness method of Aubin-Lions-Simon type.    

\medskip

  \medskip

\end{abstract}
\setcounter{section}{0}
\setcounter{equation}{0}
\section{Introduction}
We consider the following inhomogeneous (i.e., non-constant density) incompressible Navier-Stokes system with mass diffusion, which is a version of the well-known Kazhikhov-Smagulov model: 
\begin{eqnarray}\label{NS1}
 &&\left \{
\begin{array}{lll}
&\p_t \rho+ v\cdot \nabla\rho = \theta \Delta \rho \mbox{\quad  in $(0,\infty)\times\Omega $,}
\medskip\\
&\dis \p_t(\rho v)+\sum_{j=1}^3\p_{x_j} (\rho v_j v)
 -\nabla\cdot \{ \mu(\rho)(\nabla v+{}^{\rm t}\!(\nabla v)) \}   \\
&\quad   -\theta v\Delta \rho-\theta(v\cdot\nabla)\nabla\rho-\theta(\nabla\rho\cdot\nabla)v+2\theta \nabla\cdot\{\mu(\rho) \nabla\nabla(\log\rho)\}\\
&\dis \quad +\theta^2\Big\{\frac{\Delta\rho\nabla\rho}{\rho}
+  \frac{(\nabla\rho\cdot\nabla)\nabla\rho}{\rho} - \frac{|\nabla\rho|^2 \nabla\rho}{\rho^2}  \Big\} =  -\nabla p+\rho f  \mbox{\quad in $(0,\infty)\times\Omega $,}
\medskip\\
&\nabla\cdot v=0\mbox{\quad in $(0,\infty)\times\Omega$,} 
\medskip\\
&\nabla\rho\cdot\nu=0,\quad  v=0\mbox{\quad on $(0,\infty)\times\partial \Omega$},
\medskip\\
&\rho(0,\cdot)=\eta,\quad v(0,\cdot)=u\mbox{\quad in $\Omega$},
\end{array}
\right.
\end{eqnarray}
where 
\begin{itemize}
\item $\Omega\subset\R^3$ is a bounded connected open set with the smooth boundary $\p\Omega$ and its unit outer normal $\nu=\nu(x)$, 
\item $\rho=\rho(t,x)$, $v=v(t,x)=(v_1(t,x),v_2(t,x),v_3(t,x))$, $p=p(t,x)$ are the unknown density, velocity, pressure, respectively,     
\item $\mu:[0,\infty)\to(0,\infty)$ is the viscosity (a given function) depending only on the density and $\theta>0$ is the constant mass diffusivity,
\item $f=f(t,x)=(f_1(t,x),f_2(t,x),f_3(t,x))$ is a given external force and $\eta,u$ are initial data, 
\item   $\nabla :=(\p_{x_1},\p_{x_2},\p_{x_3})$, $\Delta:=\p_{x_1}^2+\p_{x_2}^2+\p_{x_3}^3$,  $\nabla v$ is the Jacobian matrix of $v$ and $x\cdot y:=x_1y_1+x_2y_2+x_3y_3$ for $x=(x_1,x_2,x_3)$ and $y=(y_1,y_2,y_3)$. 
\end{itemize}
The system with $\theta=0$, i.e., the flow without mass diffusion, is the standard inhomogeneous incompressible Navier-Stokes equations  (in that case, the Neumann boundary condition of $\rho$ is not necessary) and there is large literature on this problem: see \cite{AK}, \cite{Kazhikhov}, \cite{JL}, \cite{Kim}, \cite{Simon}, \cite{AKM}, \cite{Lions} and \cite{DM}.  

We briefly explain the origin of \eqref{NS1}.  Consider the general compressible viscous Navier-Stokes system: 
\begin{eqnarray}\label{NSc}
&&\left \{
\begin{array}{lll}
&\p_t \rho+\nabla\cdot (\rho V)=0,\\
&\rho\{\p_t V+(V\cdot\nabla)V\}-\nabla \Big\{\Big(\chi-\frac{2}{3}\mu\Big)(\nabla\cdot V))   \Big\}
-\nabla\cdot \{ \mu(\nabla V+{}^{\rm t}\!(\nabla V)) \}\\
&\quad =-\nabla q+\rho f,
\end{array}
\right.
\end{eqnarray}
where $\chi$ and $\mu$ are not necessarily constant. As constitutive laws (its physical interpretation should be the central issue in modeling), we suppose that 
\begin{eqnarray}\label{con}
\nabla\cdot V=-\theta\Delta(\log \rho)
,\quad \theta>0\mbox{ is  constant},\quad 
 \chi,\,\,\mu\mbox{ are functions of $\rho$.}
\end{eqnarray}
Introducing the new variable $v:=V+\theta\nabla(\log\rho)$
and rewriting the compressible system \eqref{NSc} with $v$, we obtain \eqref{NS1}, where a term described as the gradient of a scalar function is included in the pressure.  

Kazhikhov-Smagulov \cite{KS} (see also \cite{AKM}) demonstrated such modeling for the case with the constant viscosity and obtained weak solutions by neglecting the $\theta^2$-nonlinearity, where the system with the constant viscosity (after the above transformation) is given as 
\begin{eqnarray*}
&&\left \{
\begin{array}{lll}
&\p_t \rho+v\cdot \nabla\rho =\theta\Delta\rho,\\
& \dis \rho(\p_t v+(v\cdot\nabla)v)-\mu \Delta v-\theta(v\cdot\nabla)\nabla\rho-\theta(\nabla\rho\cdot\nabla)v\\
&\dis  +\theta^2\Big\{\frac{\Delta\rho\nabla\rho}{\rho}
+  \frac{(\nabla\rho\cdot\nabla)\nabla\rho}{\rho} - \frac{|\nabla\rho|^2 \nabla\rho}{\rho^2}  \Big\}
=-\nabla p+\rho f , 
\end{array}
\right.
\end{eqnarray*}
or equivalently 
\begin{eqnarray}\label{NS0}
&&\left \{
\begin{array}{lll}
&\p_t \rho+v\cdot \nabla\rho =\theta\Delta\rho,\\
& \dis\dis \p_t(\rho v)+\sum_{j=1}^3\p_{x_j} (\rho v_j v) 
-\mu \Delta v -\theta v\Delta \rho-\theta(v\cdot\nabla)\nabla\rho-\theta(\nabla\rho\cdot\nabla)v\\
&\dis  +\theta^2\Big\{\frac{\Delta\rho\nabla\rho}{\rho}
+  \frac{(\nabla\rho\cdot\nabla)\nabla\rho}{\rho} - \frac{|\nabla\rho|^2 \nabla\rho}{\rho^2}  \Big\}
=-\nabla p+\rho f. 
\end{array}
\right.
\end{eqnarray}
Beir\~ao da Veiga \cite{B} proved existence of local in time strong solutions of \eqref{NS0} with the $\theta^2$-nonlinearity;  Secchi \cite{S} obtained global in time strong solutions of \eqref{NS0} in the $2$-dimensional space and demonstrated asymptotics of strong solutions as $\theta\to0$. Cabrales et. al. \cite{CGG} showed existence of weak solutions of \eqref{NS0} including the $\theta^2$-nonlinearity via a numerical scheme.  Bresch et. al.  \cite{BHM} and Calgaro et. al. \cite{CEZ}  derived versions of  \eqref{NS0} without the $\theta^2$-nonlinearity by specific choice of the constants $\mu$ and $\theta$.  
Cook-Dimotakis \cite{CD} investigated \eqref{NSc} with the constant viscosity and \eqref{con} to understand Rayleigh-Taylor instability between miscible fluids. Goudon-Vasseur \cite{GV} re-analyzed Kazhikhov-Smagulov type modeling of mixture flows from a more microscopic viewpoint to propose  a generalized model.  

In the literature of mathematical analysis of  weak solutions to Kazhikhov-Smagulov type system, $L^2$-energy estimates (i.e.,  $\int_\Omega$\eqref{NS0}$\,\cdot\,v$) play the central role. In this argument, one can estimate the $\theta$-nonlinearity from above by $\frac{\theta}{2}(M-m)\norm \nabla v\norm_{L^2(\Omega)}^2$ as long as $0<m\le\rho\le M$ (see Chapter  3, Section 4 of \cite{AKM}), which should not dominate  $\mu\norm \nabla v\norm_{L^2(\Omega)}^2$ coming from $\int_\Omega\mu\Delta v\cdot v$. Neglecting the $\theta^2$-nonlinearity in \eqref{NS0}, Kazhikhov-Smagulov \cite{KS} obtained weak solutions of \eqref{NS0} under the condition 
\begin{eqnarray}\label{con2}
\frac{\theta}{2}(M-m)<\mu.
\end{eqnarray}
If we keep the $\theta^2$-nonlinearity,  it is estimated from above by 
$$C_1\theta^2 \norm  \Delta\rho\norm_{L^2(\Omega)}\norm \nabla v\norm_{L^2(\Omega)}\mbox{ \quad ($C_1$ is some constant)},$$
while the estimate of $ \norm  \Delta\rho\norm_{L^2(\Omega)}$ contains  $C_2\theta^{-1} \norm \nabla v\norm_{L^2(\Omega)}$ ($C_2$ is some constant). Hence, if \eqref{con2} is tighten to be $\theta\ll\mu$, one can still obtain  weak solutions of \eqref{NS0} with the  $\theta^2$-nonlinearity.  
    
If $\mu$ is not constant, the situation changes:  the $\theta$-nonlinearity has one more term $2\theta \nabla\cdot\{\mu(\rho) \nabla\nabla(\log\rho)\}$. 
If $\mu$ is not smooth, the $L^2$-estimate for $\int_\Omega(2\theta \nabla\cdot\{\mu(\rho) \nabla\nabla(\log\rho)\})\,\cdot\, v$  yields 
$$\theta(\max\mu|_{[m,M]})\norm  \Delta\rho\norm_{L^2(\Omega)}\norm \nabla v\norm_{L^2(\Omega)},$$
 which leads to $C_3(\max\mu|_{[m,M]})\norm \nabla v\norm_{L^2(\Omega)}^2$ ($C_3$ is some constant) together with  the above mentioned  estimate of $ \norm  \Delta\rho\norm_{L^2(\Omega)}$. 
 Then, we see that this term would dominate $(\min\mu|_{[m,M]})\norm\nabla v\norm_{L^2(\Omega)}^2$ coming from Korn's inequality applied to $(\nabla\cdot \{ \mu(\rho)(\nabla v+{}^{\rm t}\!(\nabla v)) \})\,\cdot\,v$.  Hence,  \eqref{con2}-type condition cannot control the term $2\theta \nabla\cdot\{\mu(\rho) \nabla\nabla(\log\rho)\}$. We propose an idea and hypothesis to overcome this trouble (later, we will state it more precisely): {\it Suppose that $\mu$ is $C^1$-smooth. Then, noting that $v$ is divergence-free, we have    
\begin{eqnarray*}
\nabla\cdot\{\mu(\rho)\nabla\p_{x_i}\log\rho\} &=&   \p_{x_i}\Big[\nabla\cdot \{\mu(\rho)\nabla\log\rho\}\Big]-\nabla\cdot\{\p_{x_i}(\mu(\rho))\nabla\log\rho\},\\
\int_\Omega \{\nabla \cdot(\mu(\rho)\nabla\nabla\log\rho)\}\cdot vdx&=&\sum_{i,j=1}^3\int_\Omega  
 \frac{\mu'(\rho)}{\rho}(\p_{x_i}\rho)(\p_{x_j}\rho)(\p_{x_j}v_i)dx.
 \end{eqnarray*}
 If $\mu(\rho)$ satisfies $\theta\ll\mu|_{[m,M]}$ and $\mu'|_{[m,M]}\ll\mu|_{[m,M]}$,  
the $\theta,\theta^2$-nonlinearities in \eqref{NS1} are controllable.}  

In the literature of inhomogeneous incompressible  Navier-Stokes system including Kazhikhov-Smagulov type models, Galerkin type approximation is very common to construct a weak solution. To prove convergence of the approximation, one needs $L^2$-strong convergence of the approximate velocity, which is often done by showing some equi-continuity of the approximate velocity with respect to the time variable (see, e.g., \cite{JL}).  This issue is complicated, as the time-dependency of the velocity appears from the Navier-Stokes part in the form of the time derivative of [density]$\times$[velocity].  

In this paper, we demonstrate a procedure with time-discretization and iterative treatment of linear elliptic problems, which are easily solved by the standard Lax-Milgram framework. We propose a new compactness argument for $L^2$-strong convergence of the approximate velocity, which is a version (with continuous variables) of the technique developed in the work Soga \cite{Soga2023} on mathematical analysis of a fully discrete finite difference method for inhomogeneous incompressible Navier-Stokes system without mass diffusion. The basic idea is to modify the standard Aubin-Lions lemma applied to the velocity so that the weak norm of velocity is replaced by that of [density]$\times$[velocity]; then, one can immediately evaluate the weak norm by inserting the Navier-Stokes part. More precisely, we will prove and apply the following technical lemmas:        
{\it 
Consider two sequences $\{\rho_k\}_{k\in\N}$ and $\{ v_k \}_{k\in\N}$ such that
\begin{itemize}
\item[(A1)] $\rho_k\in L^\infty([0,T];L^\infty(\Omega))$, $0<\alpha\le\rho_k(t,x)\le \beta$\,\,\, a.e. $(t,x)$ for all $k$. 
\item[(A2)] $\{\rho_k\}_{k\in\N}$, $\rho_k:[0,T]\to L^2(\Omega)$ is weakly equi-continuous in the sense that for each $\phi\in C^\infty_0(\Omega)$, $\{(\rho_k,\phi)_{L^2(\Omega)}\}_{k\in\N}$ is equi-continuous on $[0,T]$. 
\item[{(A3)}]  There exists $\rho\in L^\infty([0,T];L^\infty(\Omega))$ such that $\alpha\le\rho\le \beta$ a.e. and $\rho_k(t,\cdot)\weakto\rho(t,\cdot)\mbox{ in $L^2(\Omega)$ as $k\to\infty$ for each $t\in[0,T]$}$. 
\item[(A4)] $v_k\in L^2([0,T];H^1_{0,\sigma}(\Omega))$ and $v_k(t,\cdot)$ has a value  in  $H^1_{0,\sigma}(\Omega)$ for every $t\in[0,T]$.  
\end{itemize}
 \begin{Lemma}\label{key-lemma1}
 If $\{\rho_k\}_{k\in\N}$ satisfies (A1) and (A2), then there exists a subsequence of $\{\rho_k\}_{k\in\N}$, still denoted by $\{\rho_k\}_{k\in\N}$, that satisfies (A3).
\end{Lemma}
\begin{Lemma}\label{key-lemma2}
Suppose that $\{\rho_k \}_{k\in\N}$  and $\{v_k \}_{k\in\N}$ satisfy (A1)--(A4). Then, for each $\lambda>0$, there exists a constant $A_\lambda\ge0$ such that  for all $ t\in[0,T]$ and all $   k,l\in\N$,
\begin{eqnarray}\label{key5555}
&& \norm v_k(t,\cdot)-v_l(t,\cdot) \norm_{L^2(\Omega)^3}\le \lambda(\norm v_k(t,\cdot)\norm_{H^1(\Omega)^3}+\norm v_l(t,\cdot)\norm_{H^1(\Omega)^3} +k^{-1}+l^{-1})\\\nonumber
&&\qquad  +A_\lambda\Big( \sup_{\varphi\in S}\Big|\Big( \rho_k(t,\cdot)v_k(t,\cdot)-\rho_l(t,\cdot)v_l(t,\cdot),\varphi \Big)_{L^2(\Omega)^3}\Big| +k^{-1}+l^{-1} \Big)
 \end{eqnarray}
holds,  where $S$ is the family of all $\varphi\in C^\infty_{0,\sigma}(\Omega)$ with $\norm \varphi\norm_{W^{1,\infty}(\Omega)^3}=1$.
\end{Lemma}
} 
\noindent Note that there would be another version or generalization of Lemma \ref{key-lemma2} in different function spaces (we postpone discussions). We will see that our convergence proof based on Lemma \ref{key-lemma2} is quite similar to the reasoning for the  case with the constant density based on the Aubin-Lions-Simon approach.  

In Section 2, we state the definition of a weak solution of \eqref{NS1} and the main result. In Section 3, we give the time discretization and solve the elliptic problems. In Section 4, we prove the above technical lemmas. In Section 5, we discuss weak/strong convergence of the approximation. In Section 6, we prove the main result.  

\setcounter{section}{1}
\setcounter{equation}{0}
\section{Main result}
Let $\Omega\subset \R^3$ be a bounded connected open set with the smooth boundary $\p\Omega$, where the outer unit normal of $\p\Omega$ is denoted by $\nu=\nu(x)$, $x\in\p\Omega$. We use the following notation to indicate function spaces:
\begin{itemize} 
\item $C^r_{0}(\Omega)=C^r_{0}(\Omega;\R)$  is the family of $C^r$-functions $w:\Omega\to\R$ that are equivalently $0$ near $\p\Omega$, 
 $C^r_{0,\sigma}(\Omega):=\{v\in C^r_0(\Omega)^3\,|\,\nabla\cdot v=0\}$,  
\item $L^p(\Omega)=L^p(\Omega;\R)$; $H^r(\Omega)=H^r(\Omega;\R)$, $H^r_0(\Omega)=H^r_0(\Omega;\R)$ is the closure of $C^\infty_0(\Omega)$ with respect to the norm  $\norm \cdot\norm_{H^r(\Omega)}$, 
\item 
 $H^1_{0,\sigma}(\Omega)$  is the closure of $C^\infty_{0,\sigma}(\Omega)$ with respect to the norm $\norm\cdot\norm_{H^1(\Omega)^3}$,
\item $H^2_N(\Omega):=H^2_N(\Omega;\R)$ is the family of functions $w\in H^2(\Omega)$ such that $\nabla w\cdot \nu=0$ on $\p\Omega$ (in the trace sense).     
\end{itemize}
Throughout this paper, we suppose that 
{\it 
\begin{itemize}
\item[(H1)] $\Omega\subset \R^3$ is a bounded connected open set with the $C^4$-smooth boundary $\p\Omega$.
\item[(H2)]  The external force $f$ belongs to $L^2_{\rm loc}([0,\infty);L^2(\Omega)^3)$.
\item[(H3)] Initial data $\eta,u$ are such that $\eta\in H^1(\Omega)$ with $0<m\le \eta\le M$ and $u\in L^2(\Omega)^3$, where $m,M$ are arbitrarily given fixed constants.
\item[(H4)] The viscosity $\mu:[0,\infty)\to(0,\infty)$ is $C^1$-smooth and satisfies 
\begin{eqnarray*}
&& \mu_\ast-\frac{\theta}{2}(M-m)  - \frac{\tilde{C}_\Omega^2 M^2}{m}(\mu'{}^{\ast} +\theta)  >0,\\
&&\mu_\ast:=\min_{m\le\rho\le M}\mu(\rho), \quad \mu'{}^{\ast}:=\max_{m\le\rho\le M}\mu'(\rho), 
\end{eqnarray*}
where  the diffusivity $\theta>0$ is a constant, 
$\tilde{C}_\Omega:= 1+2C_\Omega+2C_\Omega(1+A_P^2)^\2A_P^\2$, $A_P$ is the constant appearing in Poincar\'e's inequality for $H^1(\Omega)$ and $C_\Omega$ is the constant appearing in \eqref{key} below.  
\end{itemize}}
Before stating the definition of weak solutions of \eqref{NS1}, we observe several equalities in vector calculus.   For any $\rho\in H^2_N(\Omega)\cap C^3(\bar{\Omega};\R)$ with $m\le \rho\le M$ and $v,w\in C^\infty_{0,\sigma}(\Omega)$, we have  
\begin{eqnarray*}
&&-\int_\Omega[ \nabla\cdot \{ \mu(\rho)(\nabla v+{}^{\rm t}\!(\nabla v)) \} ]\cdot wdx  =\sum_{i,j=1}^3 \int_\Omega  \mu(\rho) (\p_{x_j} v_i+\p_{x_i} v_j)  \p_{x_j} w_i dx,\\
&&  \int_\Omega [\nabla\cdot\{\mu(\rho) \nabla\nabla(\log\rho)\}]\cdot wdx
=   -\sum_{i=1}^3\int_\Omega \nabla\cdot[\p_{x_i}\{\mu(\rho)\}\nabla\log\rho]w_idx;
\end{eqnarray*}
for any $\rho\in H^2_N(\Omega)$ with $m\le \rho\le M$, $v\in H^1_{0,\sigma}(\Omega)$ and $w\in C^\infty_{0,\sigma}(\Omega)$, we have (with possibly via smooth approximation of $\rho,v$),   
\begin{eqnarray}\label{div1}
&&\sum_{i,j=1}^3 \int_\Omega  \mu(\rho) (\p_{x_j} v_i+\p_{x_i} v_j)  \p_{x_j} w_i dx\\\nonumber
&&\qquad =\2\sum_{i,j=1}^3\int_\Omega\mu(\rho)(\p_{x_j}v_i+\p_{x_i}v_j)(\p_{x_j}w_i+\p_{x_i}w_j)dx,
\\\label{div2}
&&-\int_\Omega \{v\Delta\rho +(v\cdot\nabla)\nabla\rho+(\nabla\rho\cdot\nabla)v\}\cdot wdx\\\nonumber
&&\qquad =\sum_{i,j=1}^3\int_\Omega\{(\p_{x_i}\rho)v_j(\p_{x_i}w_j)+(\p_{x_j}\rho)v_i(\p_{x_i}w_j)\}dx,
\\\label{div3}
&&-\sum_{i=1}^3\int_\Omega \nabla\cdot[\p_{x_i}\{\mu(\rho)\}\nabla\log\rho]w_idx
= \sum_{i,j=1}^3\int_\Omega\frac{\mu'(\rho)}{\rho}(\p_{x_i}\rho)(\p_{x_j}\rho)(\p_{x_j}w_i)dx,\\\label{div4}
&&\int_\Omega\Big\{\frac{\Delta\rho\nabla\rho}{\rho}
+  \frac{(\nabla\rho\cdot\nabla)\nabla\rho}{\rho} - \frac{|\nabla\rho|^2 \nabla\rho}{\rho^2}  \Big\}\cdot wdx
=-\sum_{i,j=1}^3\int_\Omega\frac{1}{\rho}(\p_{x_i}\rho)(\p_{x_j}\rho)(\p_{x_i}w_j)dx.
\end{eqnarray}

\noindent{\bf Definition.} {\it Let $T>0$ be an arbitrary terminal time. A pair of functions $\rho,v$ is said to be a  weak $[0,T]$-solution of \eqref{NS1}  with initial data  $\eta\in H^1(\Omega)$ such that $m\le\eta\le M$ and $u\in L^2(\Omega)^3$, if 
\begin{itemize}
\item $\rho\in L^2([0,T];H^2_N(\Omega))\cap L^\infty([0,T];H^1(\Omega))$ and $m\le \rho\le M$, 
\item $v\in L^2([0,T];H^1_{0,\sigma}(\Omega))\cap L^\infty([0,T];L^2(\Omega)^3)$, 
\item For any $\phi\in C^\infty([0,T]\times\Omega;\R)$ with {\rm supp}$(\phi)\subset[0,T)\times\Omega$,   
\end{itemize}
\vspace{-3mm}
\begin{eqnarray}\label{weak1}
\dis \int_0^T\int_\Omega\{ -\rho \p_t\phi + (v\cdot \nabla\rho - \theta \Delta \rho)\phi \}dxdt-\int_\Omega \eta\phi(0,\cdot)dx=0,
\end{eqnarray}
\begin{itemize}
\item For any  $\varphi\in C^\infty([0,T]\times\Omega;\R^3)$ with $\nabla\cdot\varphi=0$ and {\rm supp}$(\varphi)\subset[0,T)\times\Omega$,
\end{itemize}
\vspace{-4mm}
\begin{eqnarray}\label{weak2}
&&\int_0^T\int_\Omega\Big[ -\rho v \cdot\p_t\varphi - \sum_{j=1}^3(\rho v_j v)\cdot \p_{x_j}\varphi 
+\sum_{i,j=1}^3\Big\{\2\mu(\rho)(\p_{x_j}v_i+\p_{x_i}v_j)(\p_{x_j}\varphi_i+\p_{x_i}\varphi_j) \\\nonumber
&&\quad +\theta\Big( (\p_{x_i}\rho)v_j(\p_{x_i}\varphi_j)+(\p_{x_j}\rho)v_i(\p_{x_i}\varphi_j)  \Big)
+2\theta\frac{\mu'(\rho)}{\rho}(\p_{x_i}\rho)(\p_{x_j}\rho)(\p_{x_j}\varphi_i)\\\nonumber
&&\quad -\theta^2\frac{1}{\rho}(\p_{x_i}\rho)(\p_{x_j}\rho)(\p_{x_i}\varphi_j)\Big\}-\rho f\cdot\varphi\Big]dxdt
-\int_\Omega\eta u\cdot\varphi(0,\cdot)dx=0.
\end{eqnarray}
A pair of functions $\rho,v$ is said to be a global weak solution of \eqref{NS1}  with initial data  $\eta\in H^1(\Omega)$ such that $m\le\eta\le M$ and $u\in L^2(\Omega)^3$, if 
\begin{itemize}
\item $\rho\in L^2_{\rm loc}([0,\infty);H^2_N(\Omega))\cap L^\infty([0,\infty);H^1(\Omega))$ and $m\le \rho\le M$, 
\item $v\in L^2_{\rm loc}([0,\infty);H^1_{0,\sigma}(\Omega))\cap L^\infty([0,\infty);L^2(\Omega)^3)$, 
\item For any $T>0$, the pair $\rho|_{[0,T]}$, $v|_{[0,T]}$ is a  weak $[0,T]$-solution of \eqref{NS1}.      
\end{itemize}}
\medskip

Here is the main result: 
\begin{Thm}\label{main1}
Suppose  (H1)-(H4). Then, there exists a global weak solution of \eqref{NS1}.  
\end{Thm}
%
%
%
%
%
\noindent We state a long time behavior of weak solutions (this is obvious for classical solutions): 
\begin{Prop}\label{Remark1}
Let $\rho,v$ be a global weak solution of \eqref{NS1}. Then, it holds that 
\begin{eqnarray}\label{remark111}
&& \frac{1}{{\rm vol}(\Omega)}\int_\Omega \rho(t,\cdot) dx\equiv\bar{\eta}:=\frac{1}{{\rm vol}(\Omega)}\int_\Omega\eta dx,\mbox{ \quad $\forall\,t\ge0$},
\end{eqnarray}
and $\norm \rho(t,\cdot)-\bar{\eta}\norm_{L^2(\Omega)}$ is absolutely continuous with respect to $t$ satisfying 
\begin{eqnarray}\label{remark11}\label{remark12}
&&\norm \rho(t,\cdot)-\bar{\eta}\norm_{L^2(\Omega)}\le\norm \eta-\bar{\eta}\norm_{L^2(\Omega)} e^{-2\theta A_P^2 t},\mbox{ \quad $\forall\,t\ge0$},
\end{eqnarray}
where $A_P>0$ is the constant in Poincar\'e inequality for $H^1(\Omega)$ depending only on $\Omega$. \eqref{remark111} means ``mass conservation'' and  \eqref{remark12} shows ``mixing of fluids''. 
\end{Prop}
\begin{proof}
Let $T>0$ be an arbitrary number.  
Let $\chi$ be the indicator function of $\Omega$ ($\chi(x)\equiv1$ on $\Omega$) and $\chi_\ep\in C^\infty_0(\Omega;[0,1])$ be such that $\chi_\ep\to \chi$ in $L^2(\Omega)$ as $\ep\to0$ and $\chi_\ep(x)\equiv 1$ on $\Omega_\ep:=\{x\in\Omega\,|\,{\rm dist}(x,\p\Omega)>\ep\}$.  Taking $\phi:=g(t)\chi_\ep(x)$ with supp$(g)\subset(0,T)$ in \eqref{weak1} and sending $\ep\to0$, we have 
\begin{eqnarray*}
\int_0^T\Big(\int_\Omega\rho(t,x) dx\Big)g'(t)dt&=&\int_0^T\Big\{\int_\Omega(v\cdot\nabla\rho-\theta \Delta \rho) dx\Big\}g(t)dt\\
&=&\int_0^T\Big\{\int_{\p\Omega}(v\rho-\theta \nabla \rho)\cdot\nu dS\Big\}g(t)dt=0,
\end{eqnarray*}
where we note that $v$ is divergence-free and $\rho$ satisfies the $0$-Neumann boundary condition.  Hence, $\int_\Omega\rho(t,x) dx$ is weakly $t$-differentiable with the weak derivative equal to $0$ and $\int_\Omega\rho(t,x) dx:(0,T)\to\R$ is constant. Taking  $\phi:=g(t)\chi_\ep(x)$ with supp$(g)\subset[0,T)$ in \eqref{weak1} and sending $\ep\to0$, we find 
$$g(0)\int_\Omega\rho(t,x) dx=g(0)\int_\Omega\eta dx,$$
which implies \eqref{remark111}. 

Set $\bar{\rho}:=\rho-  \frac{1}{{\rm vol}(\Omega)}\int_\Omega \rho(t,\cdot) dx$. We see that $\bar{\rho}= \rho-\bar{\eta}$ and $\bar{\rho}$ satisfies for any test function $\phi$, 
\begin{eqnarray}\label{wedddd}
 \int_0^T\int_\Omega\{ -\bar{\rho} \p_t\phi + (v\cdot \nabla\bar{\rho} - \theta \Delta \bar{\rho})\phi \}dxdt-\int_\Omega (\eta-\bar{\eta})\phi(0,\cdot)dx=0. 
 \end{eqnarray} 
In particular, taking $\phi$ such that supp$(\phi)\subset(0,T)\times\Omega$, we see  that $\bar{\rho}$ is weakly $t$-partial differentiable with 
$$\p_t\bar{\rho}=-v\cdot\nabla\bar{\rho}+\theta\Delta\bar{\rho}\in L^1([0,T];L^1(\Omega)).$$
Hence, we have $\bar{\rho}(t,\cdot)=\bar{\rho}(0,\cdot)+\int_0^t \{-v(s,\cdot)\cdot\nabla\bar{\rho}(s,\cdot)+\theta\Delta\bar{\rho}(s,\cdot)\}ds$, where $\bar{\rho}(0,\cdot)$ is determined so that $\bar{\rho}(t,\cdot)$ satisfies \eqref{wedddd}, i.e., $\bar{\rho}(0,\cdot)=\eta-\bar{\eta}$. 

We claim that $\bar{\rho}^2$ is weakly $t$-partial differentiable with 
\begin{eqnarray}\label{2121}
\frac{1}{2}\p_t(\bar{\rho}^2)=(\p_t\bar{\rho})\bar{\rho}=-(v\cdot\nabla\bar{\rho})\bar{\rho}+\theta(\Delta\bar{\rho})\bar{\rho}\in L^1([0,T];L^1(\Omega)).
\end{eqnarray}
In fact, let $\beta_\delta$ be the standard mollifier in $\R_t\times\R^3_x$  with the parameter $\delta>0$ and let $\bar{\rho}_\delta:=\beta_\delta\ast\bar{\rho}$; for each test function $\phi$ such that supp$(\phi)\subset(0,T)\times\Omega$, we have $\norm  \bar{\rho}_\delta-\bar{\rho}\norm_{L^1({\rm supp}(\phi))}\to0$, $\norm  \p_t\bar{\rho}_\delta-\p_t\bar{\rho}\norm_{L^1({\rm supp}(\phi))}\to0$  as $\delta\to0$; taking a subsequence of $\{\bar{\rho}_\delta\}_{\delta>0}$, still denoted by the same symbol, we see that $(\p_t\bar{\rho})(\bar{\rho}_\delta-\bar{\rho})\phi \to0$ a.e. on supp$(\phi)$ with $|(\p_t\bar{\rho})(\bar{\rho}_\delta-\bar{\rho})\phi|\le 2M|(\p_t\bar{\rho})\phi|$;  Lebesgue's dominated convergence theorem implies \eqref{2121}, i.e.,   
\begin{eqnarray*}
&&\int_0^T\int_\Omega \2\bar{\rho}_\delta^2\p_t\phi dxdt=-\int_0^T\int_\Omega (\p_t\bar{\rho}_\delta)\bar{\rho}_\delta\phi dxdt\\
&&\quad=-\int_0^T\int_\Omega (\p_t\bar{\rho}_\delta-\p_t\bar{\rho})\bar{\rho}_\delta\phi dxdt-\int_0^T\int_\Omega\p_t\bar{\rho}(\bar{\rho}_\delta-\bar{\rho})\phi dxdt -\int_0^T\int_\Omega (\p_t\bar{\rho})\bar{\rho}\phi dxdt\\
&&\qquad \to \int_0^T\int_\Omega \2\bar{\rho}^2\p_t\phi dxdt=-\int_0^T\int_\Omega (\p_t\bar{\rho})\bar{\rho}\phi dxdt\quad\mbox{as $\delta\to0$.}
\end{eqnarray*} 
 Taking $\phi:=g(t)\chi_\ep(x)$ with supp$(g)\subset(0,T)$ in the previous equality, we have with \eqref{2121}, 
\begin{eqnarray*}
&& \int_0^T\Big(\int_\Omega \frac{1}{2}(\bar{\rho}^2)\chi_\ep(x)dx\Big)g'(t)dt
 = -\int_0^T\int_\Omega \frac{1}{2}(\p_t\bar{\rho})\bar{\rho}\chi_\ep(x)g(t)dxdt\\
 &&\quad =\int_0^T\int_\Omega\Big\{(v\cdot\nabla\bar{\rho})\bar{\rho}-\theta(\Delta\bar{\rho})\bar{\rho}\Big\} \chi_\ep(x) g(t)dxdt.
\end{eqnarray*} 
Sending $\ep\to0$ and noting again that $v$ is divergence-free, we obtain with integration by parts 
\begin{eqnarray*}
\int_0^T\2\norm \bar{\rho}(t,\cdot)\norm_{L^2(\Omega)}^2g'(t)dt&=&  \int_0^T\int_\Omega \Big\{(v\cdot\nabla\bar{\rho})\bar{\rho}-\theta(\Delta\bar{\rho})\bar{\rho}\Big\} g(t)dxdt\\
&=&\int_0^T\theta\norm \nabla\bar{\rho}(t,\cdot)\norm_{L^2(\Omega)}^2g(t)dt.
\end{eqnarray*}
Hence, $\norm \bar{\rho}(t,\cdot)\norm_{L^2(\Omega)}^2$ is weakly differentiable on $[0,T]$ with the weak derivative equal to $-2\theta\norm \nabla\bar{\rho}(t,\cdot)\norm_{L^2(\Omega)}^2$. 
Therefore, $\norm \bar{\rho}(t,\cdot)\norm_{L^2(\Omega)}^2$ is absolutely continuous on $[0,T]$ and satisfies  with Poincar\'e inequality,  
$$\frac{d}{dt}\norm \bar{\rho}(t,\cdot)\norm_{L^2(\Omega)}^2=- 2\theta \norm \nabla\bar{\rho}(s,\cdot)\norm_{L^2(\Omega)}^2\le - 2\theta A_P^2 \norm \bar{\rho}(t,\cdot)\norm_{L^2(\Omega)}^2,\quad\mbox{a.e. } t\in[0,T].$$
Since $T>0$ is arbitrary, we conclude \eqref{remark12}. 
\end{proof}
\noindent Note that  Beir\~ao da Veiga \cite{B} showed this kind of asymptotic behavior of a strong solution of \eqref{NS0}$|_{f=0}$ with initial data sufficiently close to the constant solution $\bar{\eta},0$.  
%
\setcounter{section}{2}
\setcounter{equation}{0}
\section{Time-discretization and elliptic problems}

Let  $0<\tau\le\2$ be the time-discretization parameter that will be sent to $0$ at the end. For each $n\in\N\cup\{0\}$, define $f^{n+1}$ as 
\begin{eqnarray}\label{333fff}
f^{n+1}:=\tau^{-1}\int_{\tau n}^{\tau n+\tau}f(t,\cdot)dt.
\end{eqnarray}
For any initial data $\eta\in H^1(\Omega)$ satisfying $m\le \eta\le M$,  we can find $\eta_\tau\in C^\infty(\bar{\Omega})$ such that 
$$m\le \eta_\tau\le M,\quad \norm \eta_\tau-\eta\norm_{H^1(\Omega)}\le \tau.$$
We  inductively introduce a series of linear elliptic problems with $n\in\N\cup\{0\}$ as 
\begin{eqnarray}\nonumber
&&\qquad\,\, \rho^0:=\eta_\tau,\quad  v^0:=u,\quad  v^0_\tau:=0,\\
\label{ellip1}
&&\left \{
\begin{array}{lll}
&\dis  \frac{\rho^{n+1}-\rho^n}{\tau}+v_\tau^n\cdot\nabla\rho^{n+1}=\theta\Delta \rho^{n+1}\mbox{\,\,\, in $\Omega$, }\medskip\\
&\nabla\rho^{n+1}\cdot\nu=0\mbox{\,\,\, on $\p\Omega$,\quad $\rho^{n+1}\in H^2_N(\Omega)$,}
\end{array}
\right.\\\label{ellip2}
&&\left \{
\begin{array}{lll}
&\dis \frac{\rho^{n+1}v^{n+1}-\rho^{n}v^n}{\tau}+\sum_{j=1}^3\p_{x_j} (\rho^{n+1} v^n_{\tau j} v^{n+1})
 -\nabla\cdot \{ \mu(\rho^{n+1})(\nabla v^{n+1}+{}^{\rm t}\!(\nabla v^{n+1})) \}   \\
&\qquad  -\theta v^{n+1}\Delta \rho^{n+1} -\theta(v^{n+1}\cdot\nabla)\nabla\rho^{n+1}-\theta(\nabla\rho^{n+1}\cdot\nabla)v^{n+1}\\
&\dis \qquad+2\theta \nabla\cdot\{\mu(\rho^{n+1}) \nabla\nabla(\log\rho^{n+1})\}\\
&\dis \qquad +\theta^2\Big\{\frac{\Delta\rho^{n+1}\nabla\rho^{n+1}}{\rho^{n+1}}
+   \frac{(\nabla\rho^{n+1}\cdot\nabla)\nabla\rho^{n+1}}{\rho^{n+1}} - \frac{|\nabla\rho^{n+1}|^2 \nabla\rho^{n+1}}{\rho^{n+1}\rho^{n+1}}  \Big\} \medskip\\
&\quad=  \rho^{n+1} f^{n+1}\mbox{\,\,\, in $\Omega$},\medskip\\
&v^{n+1}=0\mbox{\,\,\, on $\p\Omega$,\quad $v^{n+1}\in H^1_{0,\sigma(\Omega)}$},
\end{array}
\right.\\\nonumber
&&\qquad\,\, \mbox{$v^n_\tau$ with $n\ge1$ stands for a $C^\infty_{0,\sigma}(\Omega)$-approximation of  $v^n\in H^1_{0,\sigma}(\Omega)$}\\\nonumber
&&\qquad\,\,\mbox{such that $ \norm v^n_\tau-v^n\norm_{H^1(\Omega)^3}\le \tau$.}
\end{eqnarray} 
 Note that \eqref{ellip1} and \eqref{ellip2} are not coupled.    
 
We solve \eqref{ellip1} and \eqref{ellip2} from $n=0$ in the sense of weak solutions, assuming that $\rho^n\in C^2(\bar{\Omega})$ and $v^n\in H^1_{0,\sigma}(\Omega)$ are given, while we demonstrate regularity arguments for \eqref{ellip1}.   
\begin{Prop}\label{P31}
 There exists a unique solution $\rho^{n+1}\in H^4(\Omega)\cap H^2_N(\Omega)$ of  \eqref{ellip1}, where Sobolev embedding theorem implies $\rho^{n+1}\in C^2(\bar{\Omega})$ and hence, the Neumann boundary condition is satisfied in the classical sense.  Furthermore, $\rho^{n+1}$ satisfies 
 \begin{eqnarray*}\label{P31e}
\norm \rho^{n+1}\norm_{H^2(\Omega)}^2\le C_\Omega(\norm\Delta\rho^{n+1}\norm_{L^2(\Omega)}^2+\norm \rho^{n+1}\norm_{H^1(\Omega)}^2),
\end{eqnarray*}
where $C_\Omega$ is a constant depending only on $\Omega$.     
\end{Prop}
 \begin{proof}
Define the Lax-Milgram bilinear form $B$ for the weak formulation of  \eqref{ellip1} as 
 \begin{eqnarray*}
&& B:H^1(\Omega)\times H^1(\Omega)\to\R,\\
&&B(\rho,w):=\theta(\nabla \rho,\nabla w)_{L^2(\Omega)^3}+ \tau^{-1}( \rho, w)_{L^2(\Omega)}+(v^n_\tau\cdot\nabla \rho, w)_{L^2(\Omega)},
 \end{eqnarray*}
 where $(\cdot,\cdot)_{L^2(\Omega)},(\cdot,\cdot)_{L^2(\Omega)^3}$ stand for the inner product of $L^2(\Omega),L^2(\Omega)^3$, respectively. For any $\rho,w\in H^1(\Omega)$, we have  
 \begin{eqnarray*}
| B(\rho,w)|&\le& \Big(\theta+\tau^{-1}+\max_{x\in\Omega}|v^n_\tau|\Big)\norm\rho\norm_{H^1(\Omega)}\norm w\norm_{H^1(\Omega)},\\
 B(\rho,\rho)&=& \theta\norm\nabla\rho\norm_{L^2(\Omega)}+\tau^{-1}\norm\rho\norm_{L^2(\Omega)}\ge \min\{\theta,\tau^{-1} \}\norm\rho\norm_{H^1(\Omega)}.
 \end{eqnarray*}
 Hence, Lax-Milgram theorem yields a unique solution $\rho^{n+1}\in H^1(\Omega)$ of 
 $$B(\rho,w)=(g,w),\quad \forall\,w\in H^1(\Omega),\,\,\,\,\,g:=\tau^{-1}\rho^n\in L^2(\Omega).$$
If we set $\tilde{g}:=\theta^{-1}\{g-v^n_\tau\cdot\nabla\rho^{n+1}+(1-\tau^{-1})\rho^{n+1}\}\in L^2(\Omega)$, we see that $\rho^{n+1}$ is a solution of 
\begin{eqnarray}\label{311}
(\nabla \rho,\nabla w)_{L^2(\Omega)^3}+ ( \rho, w)_{L^2(\Omega)}=(\tilde{g},w)_{L^2(\Omega)},\quad\forall\,w\in H^1(\Omega),  
\end{eqnarray} 
where \eqref{311} is the weak formulation of 
\begin{eqnarray}\label{3112}
-\Delta\rho+\rho=\tilde{g}\mbox{ in $\Omega$}, \quad \nabla\rho\cdot\nu=0\mbox{ on $\p\Omega$}.  
\end{eqnarray}
The classical results on the problem \eqref{311} (see, e.g., Chapter 5, Section 7 of \cite{Taylor}) state that {\it for each $\tilde{g}\in L^2(\Omega)$,  \eqref{311} admits a unique solution $\rho\in H^1(\Omega)$; 
$\rho$ in fact belongs to  $H^2_N(\Omega)$ and satisfies \eqref{3112} almost everywhere; 
if $\p\Omega$ is $C^{2+r}$-smooth and $\tilde{g}\in H^r(\Omega)$, $\rho$ belongs to $H^{2+r}(\Omega)$;  there exists a constant $C_\Omega$  depending only on $\Omega$ such that 
\begin{eqnarray}\label{key}
\norm \tilde{\rho}\norm_{H^2(\Omega)}^2\le C_\Omega(\norm\Delta\tilde{\rho}\norm_{L^2(\Omega)}^2+\norm \tilde{\rho}\norm_{H^1(\Omega)}^2),\quad \forall\,\tilde{\rho}\in H^2_N(\Omega).
\end{eqnarray}}
Applying this statement to $\rho^{n+1}$, we conclude our assertion.  
 \end{proof}
\begin{Prop}\label{P32}
It holds that  $\int_\Omega\rho^{n+1}dx=\int_\Omega\rho^ndx$. Furthermore, 
if $m\le \rho^n\le M$,  the solution $\rho^{n+1}$ of  \eqref{ellip1} satisfies   $m\le \rho^{n+1}\le M$. 
\end{Prop}
\begin{proof}
Since $\rho^{n+1}$ is a classical solution of  \eqref{ellip1} and $v^n_\tau\in C^\infty_{0,\sigma}(\Omega)$, we have 
$$\tau^{-1}\int_\Omega(\rho^{n+1}-\rho^n)dx=\int_{\p\Omega}(-v^n_\tau\rho^n+\theta \nabla\rho^{n+1})\cdot\nu dS=0.$$ 
\indent Observe that     
\begin{eqnarray*}
&&0=\rho^{n+1}-\rho^n+\tilde{v}^n\cdot\nabla\rho^{n+1}\tau- \tau\theta \Delta \rho^{n+1}, \\
&&0=(m-\rho^{n+1})-(m-\rho^n)+\tilde{v}^n\cdot\nabla(m-\rho^{n+1})\tau- \tau\theta \Delta (m-\rho^{n+1}).
\end{eqnarray*}
Hence $\rho:=m-\rho^{n+1}\in C^2(\bar{\Omega})$ satisfies 
\begin{eqnarray*}
-\tau\theta \Delta\rho+\rho+\tilde{v}^n\cdot\nabla\rho\tau=m-\rho^n\le 0\mbox{ in $\Omega$},\quad \nabla\rho\cdot\nu=0\mbox{ on $\p\Omega$}.
\end{eqnarray*}
We want to prove that $\rho\le 0$ in $\Omega$ (then $\rho\le 0$ on $\bar{\Omega}$ due to continuity). Suppose that $\max_{x\in\bar{\Omega}}\rho(x)>0$.  
If there is a point of $\Omega$ that attains the maximum,  then the strong maximum principle (see e.g., Chapter 6, Section 6.4.3 of \cite{Evans}) implies that $\rho$ is constant within $\Omega$ to yield $\rho\equiv m-\rho^n\le 0$, which is a contradiction.   Hence, the maximum is attained only on $\p\Omega$, i.e., there exists $x_0\in\p\Omega$ such that $\rho(x)< \rho(x_0)$ for all $x\in \Omega$. Then, Hopf's lemma (see e.g., Chapter 6, Section 6.4.3 of \cite{Evans}) implies $\nabla\rho (x_0)\cdot\nu(x_0)>0$, which is a contradiction. Thus, we conclude $\max_{x\in\bar{\Omega}}\rho(x)\le0$ to have $m\le \rho^{n+1}$. The same reasoning with $\rho:=\rho^{n+1}-M$ yields $\rho^{n+1}\le M$.
\end{proof}
\begin{Prop}\label{P33}
The solution $\rho^{n+1}$ of  \eqref{ellip1} satisfies the following estimates:
\begin{eqnarray}\label{P331}
&&\norm \rho^{n+1}\norm_{L^2(\Omega)}\le \norm \rho^{n}\norm_{L^2(\Omega)} \le \norm \rho^{0}\norm_{L^2(\Omega)} ,\\\label{P332}
&&\norm \rho^{n+1}\norm_{L^2(\Omega)}^2+\theta\norm \nabla \rho^{n+1}\norm_{L^2(\Omega)^3}^2 \tau\le \norm \rho^{n}\norm_{L^2(\Omega)}^2,\\\label{P333}
&&\theta\norm\Delta \rho^{n+1}\norm_{L^2(\Omega)}^2\tau +\norm \nabla \rho^{n+1}\norm_{L^2(\Omega)}^2 \le  \norm \nabla \rho^{n}\norm_{L^2(\Omega)}^2 + \frac{\tilde{C}_\Omega^2 M^2}{\theta}  \norm \nabla v^n_\tau\norm_{L^2(\Omega)^{3\times3}}^2 \tau,
\end{eqnarray}  
where $\tilde{C}_\Omega=1+2C_\Omega+2C_\Omega(1+A_P^2)^\2A_P^\2$ and $ \norm \nabla w\norm_{L^2(\Omega)^{3\times3}}^2:=\sum_{i,j=1}^3\norm \p_{x_j}w_i\norm_{L^2(\Omega)}^2$ for $w\in H^1(\Omega)^3$.
\end{Prop}
\begin{proof}
Multiplying $\rho^{n+1}$ to the first line of \eqref{ellip1} and integrating it over $\Omega$, we obtain
\begin{eqnarray}\label{P3312}
&&\norm \rho^{n+1}\norm_{L^2(\Omega)}^2+\theta\norm\nabla\rho^{n+1}\norm_{L^2(\Omega)^3}^2\tau=\int_\Omega\rho^n\rho^{n+1}dx
\le\norm \rho^{n}\norm_{L^2(\Omega)}\norm \rho^{n+1}\norm_{L^2(\Omega)},
\end{eqnarray}
which implies \eqref{P331}. Applying \eqref{P331} to \eqref{P3312}, we see \eqref{P332}. 
Multiplying $\Delta\rho^{n+1}$ to the first line of \eqref{ellip1} and integrating it over $\Omega$, we have 
\begin{eqnarray*}
&&\theta\norm\Delta\rho^{n+1}\norm_{L^2(\Omega)}^2\tau
+\norm \nabla\rho^{n+1}\norm_{L^2(\Omega)^3}^2=\int_\Omega  \nabla\rho^{n}\cdot \nabla\rho^{n+1}dx-\tau\int_\Omega (v^n_\tau\cdot  \nabla\rho^{n+1})\Delta\rho^{n+1}dx\\
&&\quad \le \2\norm\nabla\rho^{n}\norm_{L^2(\Omega)^3}^2+\2\norm\nabla\rho^{n+1}\norm_{L^2(\Omega)^3}^2-\tau\int_\Omega (v^n_\tau\cdot  \nabla\rho^{n+1})\Delta\rho^{n+1}dx.
\end{eqnarray*}
Observe that 
\begin{eqnarray*}
I&:=&\int_\Omega (v^n_\tau\cdot  \nabla\rho^{n+1})\Delta\rho^{n+1}dx\\
&=&\int_\Omega\nabla\cdot\{(v^n_\tau\cdot  \nabla\rho^{n+1})\nabla\rho^{n+1}\}dx-\int_\Omega \{\nabla (v^n_\tau\cdot  \nabla\rho^{n+1})\}\cdot\nabla\rho^{n+1}dx\\
&=&-\int_\Omega \{\nabla (v^n_\tau\cdot  \nabla\rho^{n+1})\}\cdot\nabla\rho^{n+1}dx\\
&=& -\sum_{i,j=1}^3\int_\Omega  (\p_{x_j}v^n_{\tau i})(\p_{x_i}\rho^{n+1})(\p_{x_j}\rho^{n+1})dx- \sum_{i,j=1}^3\int_\Omega  (v^n_{\tau i})(\p_{x_j}\p_{x_i}\rho^{n+1})(\p_{x_j}\rho^{n+1})dx\\
&=& -\sum_{i,j=1}^3\int_\Omega  (\p_{x_j}v^n_{\tau i})(\p_{x_i}\rho^{n+1})(\p_{x_j}\rho^{n+1})dx- \2\sum_{i,j=1}^3\int_\Omega  (v^n_{\tau i})\p_{x_i}\{(\p_{x_j}\rho^{n+1})^2\}dx\\
&=&- \sum_{i,j=1}^3\int_\Omega  (\p_{x_j}v^n_{\tau i})(\p_{x_i}\rho^{n+1})(\p_{x_j}\rho^{n+1})dx+ \2\int_\Omega  (\nabla\cdot v^n_{\tau})|\nabla\rho^{n+1}|^2dx\\
&=& -\sum_{i,j=1}^3\int_\Omega  (\p_{x_j}v^n_{\tau i})(\p_{x_i}\rho^{n+1})(\p_{x_j}\rho^{n+1})dx,\\
|I|&\le& \norm \nabla v^n_\tau\norm_{L^2(\Omega)^{3\times3}}\norm |\nabla\rho^{n+1}|^2\norm_{L^2(\Omega)}.
\end{eqnarray*}
Note that, if we use $(v^n_\tau\cdot \nabla\rho^{n})$ instead of  $(v^n_\tau\cdot \nabla\rho^{n+1})$ in \eqref{ellip1},  the second derivative of $\rho^n$ or $\rho^{n+1}$ would remain  in the integral of $(v^n_\tau\cdot \nabla\rho^{n})\Delta\rho^{n+1}$, which would cause a serious trouble.     
If we show the estimate 
\begin{eqnarray}\label{3key2}
&&\norm |\nabla\rho|^2\norm_{L^2(\Omega)}\le \tilde{C}_\Omega \norm \rho\norm_{L^\infty(\Omega)}\norm \Delta\rho\norm_{L^2(\Omega)},\quad\forall\,\rho\in H^2_N(\Omega)\cap C^2(\bar{\Omega}), \\\nonumber 
&&\tilde{C}_\Omega:=1+2C_\Omega+2C_\Omega(1+A_P^2)^\2A_P^\2,
\end{eqnarray}
we find that 
\begin{eqnarray*}
|I|&\le& \tilde{C}_\Omega M \norm \nabla v^n_\tau\norm_{L^2(\Omega)^{3\times3}} \norm \Delta\rho\norm_{L^2(\Omega)}\\
&\le& \frac{\theta}{2}  \norm \Delta\rho\norm_{L^2(\Omega)}^2+\frac{\tilde{C}_\Omega^2 M^2}{2\theta}\norm \nabla v^n_\tau\norm_{L^2(\Omega)^{3\times3}}^2,
\end{eqnarray*}
which leads to \eqref{P333}. We prove \eqref{3key2}. Observe that for any $\rho\in H^2_N(\Omega)\cap C^2(\bar{\Omega})$,
\begin{eqnarray*}
J^2&:=&\norm |\nabla\rho|^2\norm_{L^2(\Omega)}^2=\int_{\Omega} |\nabla \rho|^4dx=\sum_{i,j=1}^3\int_\Omega (\p_{x_i}\rho)(\p_{x_i}\rho)(\p_{x_j}\rho)(\p_{x_j} \rho) dx\\
&=&\sum_{i,j=1}^3\Big[-\int_\Omega \rho\p_{x_i} \{(\p_{x_i} \rho)(\p_{x_j} \rho)(\p_{x_j} \rho)\} dx+\int_{\p\Omega} \rho(\p_{x_i} \rho)(\p_{x_j} \rho)(\p_{x_j} \rho)\nu^i dS\Big]\\
&=&-\sum_{i,j=1}^3\int_\Omega \rho (\p_{x_i}^2 \rho)(\p_{x_j} \rho)(\p_{x_j} \rho) dx -\sum_{i,j=1}^3\int_\Omega 2\rho (\p_{x_i} \rho)(\p_{x_j} \rho) (\p_{x_i}\p_{x_j} \rho) dx\\
&=&-\sum_{j=1}^3\int_\Omega\rho (\Delta \rho)( \p_{x_j} \rho)(\p_{x_j} \rho) dx - 2\sum_{i,j=1}^3\int_\Omega \rho (\p_{x_i} \rho)(\p_{x_j} \rho)( \p_{x_i}\p_{x_j} (\rho-\bar{\rho})) dx\\
&\le& \norm \rho \norm_{L^\infty(\Omega)}  \norm \Delta \rho\norm_{L^2(\Omega)} J
+2\norm \rho \norm_{L^\infty(\Omega)} \norm \rho -\bar{\rho}\norm_{H^2(\Omega)} J ,\,\,\,\,\,\,\bar{\rho}:={\rm vol}(\Omega)^{-1}\int_\Omega\rho dx.
\end{eqnarray*}
Hence, we see that 
\begin{eqnarray*}
J&\le&  \norm \rho \norm_{L^\infty(\Omega)}  \norm \Delta \rho\norm_{L^2(\Omega)} +2\norm \rho\norm_{L^\infty(\Omega)}\norm \rho -\bar{\rho}\norm_{H^2(\Omega)}\\
&\le&  \norm \rho \norm_{L^\infty(\Omega)}  \{\norm \Delta \rho\norm_{L^2(\Omega)} +2C_\Omega(\norm \Delta \rho\norm_{L^2(\Omega)}+\norm  \rho-\bar{\rho}\norm_{H^1(\Omega)})\} \mbox{ \,\,\,\,(due to \eqref{key})}.
\end{eqnarray*}
Poincar\'e's inequality implies that 
\begin{eqnarray*}
\norm  \rho-\bar{\rho}\norm_{H^1(\Omega)}^2
&=&\norm  \nabla\rho\norm_{L^2(\Omega)}^2+\norm  \rho-\bar{\rho}\norm_{L^2(\Omega)}^2\le (1+A_P^2)\norm  \nabla\rho\norm_{L^2(\Omega)}^2,\\
\norm  \nabla\rho\norm_{L^2(\Omega)}^2&=& \int_\Omega \nabla(\rho-\bar{\rho})\cdot \nabla(\rho-\bar{\rho}) dx=- \int_\Omega (\rho-\bar{\rho})(\Delta\rho) dx\\
&\le&  \norm  \rho-\bar{\rho}\norm_{L^2(\Omega)}\norm  \Delta\rho\norm_{L^2(\Omega)} = A_P\norm  \nabla\rho\norm_{L^2(\Omega)} \norm \Delta\rho \norm_{L^2(\Omega)},\\
\norm  \nabla\rho\norm_{L^2(\Omega)}&\le& A_P\norm \Delta\rho \norm_{L^2(\Omega)},\quad  \norm  \rho-\bar{\rho}\norm_{H^1(\Omega)}^2\le   (1+A_P^2)A_P\norm \Delta\rho \norm_{L^2(\Omega)}^2.
\end{eqnarray*}
Hence, we obtain 
\begin{eqnarray*}
J&\le& \{1+2C_\Omega+2C_\Omega(1+A_P^2)^\2A_P^\2\} \norm \rho \norm_{L^\infty(\Omega)}  \norm \Delta \rho\norm_{L^2(\Omega)}.
\end{eqnarray*}
\end{proof}
Next, we solve \eqref{ellip2} weakly in $H^1_{0,\sigma}(\Omega)$. 
Due to \eqref{div1} and \eqref{div2}, the Lax-Milgram bilinear form $\tilde{B}$ of \eqref{ellip2} is given as 
\begin{eqnarray*}
&&\tilde{B}:H^1_{0,\sigma}(\Omega)\times H^1_{0,\sigma}(\Omega)\to \R,\\
&& \tilde{B}(v,w):=\2\sum_{i,j=1}^3\int_\Omega\mu(\rho^{n+1})(\p_{x_j}v_i+\p_{x_i}v_j)(\p_{x_j}w_i+\p_{x_i}w_j)dx\\
&&\quad +\theta\sum_{i,j=1}^3\int_\Omega\{(\p_{x_i}\rho^{n+1})v_j(\p_{x_i}w_j)+(\p_{x_j}\rho^{n+1})v_i(\p_{x_i}w_j)  \}dx\\
&&\quad -\sum_{j=1}^3\int_\Omega \rho^{n+1}v^n_{\tau j} v\cdot(\p_{x_j}w)dx+\tau^{-1}\int_\Omega\rho^{n+1}v\cdot wdx.
\end{eqnarray*}  
Due to \eqref{div3} and \eqref{div4}, the weak form of \eqref{ellip2} becomes 
\begin{eqnarray}\label{weak31}
&&\tilde{B}(v,w)=g(w),\quad\forall\,w\in H^1_{0,\sigma}(\Omega),\\\nonumber
&&g(w):=\tau^{-1}\int_\Omega\rho^{n}v^n\cdot wdx
- \theta\sum_{i,j=1}^3\int_\Omega\frac{\mu'(\rho^{n+1})}{\rho^{n+1}}(\p_{x_i}\rho^{n+1})(\p_{x_j}\rho^{n+1})(\p_{x_j}w_i)dx\\\nonumber
&&\qquad\qquad  +\theta^2\sum_{i,j=1}^3\int_\Omega\frac{1}{\rho^{n+1}}(\p_{x_i}\rho^{n+1})(\p_{x_j}\rho^{n+1})(\p_{x_i}w_j)dx+\int_\Omega \rho^{n+1}f^{n+1}\cdot wdx.
\end{eqnarray}
\begin{Prop}\label{P34}
There exists a unique solution $v^{n+1}\in H^1_{0,\sigma}(\Omega)$ of \eqref{weak31}, which satisfies
\begin{eqnarray*}
&&\Big\{ \mu_\ast-\frac{\theta}{2}(M-m)  - \frac{\tilde{C}_\Omega^2 M^2}{2m}(\mu'{}^{\ast} +\theta)  \Big\}\norm \nabla v^{n+1}\norm_{L^2(\Omega)^{3\times3}}^2\tau  \\
&&\quad - \frac{\tilde{C}_\Omega^2 M^2}{2m}(\mu'{}^{\ast} +\theta)\norm \nabla v^{n}_\tau\norm_{L^2(\Omega)^{3\times3}}^2\tau\\
&&\quad +  \2\norm\sqrt{\rho^{n+1}}v^{n+1}\norm_{L^2(\Omega)^3}^2 
 + \frac{\theta}{2m}( \mu'{}^{\ast} +\theta) \norm\nabla \rho^{n+1}\norm_{L^2(\Omega)^3}^2 \\
&&\quad\quad\le \2\norm\sqrt{\rho^{n}}v^n\norm_{L^2(\Omega)^3}^2 
+ \frac{\theta}{2m}( \mu'{}^{\ast} +\theta) \norm\nabla \rho^{n}\norm_{L^2(\Omega)^3}^2 \\
&&\qquad\qquad  + \norm\sqrt{\rho^{n+1}}f^{n+1}\norm_{L^2(\Omega)^3} \norm\sqrt{\rho^{n+1}}v^{n+1}\norm_{L^2(\Omega)^3}\tau.
\end{eqnarray*} 
\end{Prop}
\begin{proof}
Since $\rho^{n+1}\in C^2(\bar{\Omega})$ and $v^n_\tau\in C^\infty_{0,\sigma}(\Omega)$, there is a constant $C$ such that 
$$|\tilde{B}(v,w)|\le C\norm v\norm_{H^1(\Omega)^3}\norm w\norm_{H^1(\Omega)^3},\quad\forall\,v,w\in H^1_{0,\sigma}(\Omega).$$
We estimate $\tilde{B}(v,v)$ from the below. By Korn's inequality, we have 
\begin{eqnarray*}
\Big|\2\sum_{i,j=1}^3\int_\Omega\mu(\rho^{n+1})(\p_{x_j}v_i+\p_{x_i}v_j)^2dx\Big|\ge \mu_\ast \norm\nabla v\norm_{L^2(\Omega)^{3\times3}}^2.
\end{eqnarray*}
Observe that with \eqref{ellip1} and  $C^\infty_{0,\sigma}(\Omega)$-smooth approximation of $v\in H^1_{0,\sigma}(\Omega)$, 
\begin{eqnarray*}
 &&-\sum_{j=1}^3\int_\Omega \rho^{n+1}v^n_{\tau j} v\cdot(\p_{x_j}v)dx+\tau^{-1}\int_\Omega\rho^{n+1}v\cdot vdx\\
 &&\quad =\2\int_\Omega \tau^{-1}\rho^{n+1}|v|^2dx+\2\int_\Omega\{ \tau^{-1}\rho^{n+1}+(v^n_\tau\cdot\nabla\rho^{n+1})\}|v|^2 dx\\
&&\quad  =\2\int_\Omega \tau^{-1}\rho^{n+1}|v|^2dx+\2\int_\Omega (\tau^{-1}\rho^n+\theta\Delta\rho^{n+1})|v|^2dx\\
 &&\quad =\2\int_\Omega \tau^{-1}\rho^{n+1}|v|^2dx+\2\int_\Omega \tau^{-1}\rho^n|v|^2dx-\theta\sum_{i,j=1}^3\int_\Omega(\p_{x_i}\rho^{n+1})v_j(\p_{x_i}v_j)dx.
\end{eqnarray*}
Hence, we have 
\begin{eqnarray*}
&&\theta\sum_{i,j=1}^3\int_\Omega\{(\p_{x_i}\rho^{n+1})v_j(\p_{x_i}v_j)+(\p_{x_j}\rho^{n+1})v_i(\p_{x_i}v_j)  \}dx\\
&&\qquad -\sum_{j=1}^3\int_\Omega \rho^{n+1}v^n_{\tau j} v\cdot(\p_{x_j}w)dx+\tau^{-1}\int_\Omega\rho^{n+1}v\cdot v\\
&&\quad=\2\int_\Omega \tau^{-1}\rho^{n+1}|v|^2dx+\2\int_\Omega \tau^{-1}\rho^n|v|^2dx+\theta\sum_{i,j=1}^3\int_\Omega(\p_{x_j}\rho^{n+1})v_i(\p_{x_i}v_j)dx\\
&&\quad \ge \tau^{-1}m\norm v\norm_{L^2(\Omega)^3}^2-\theta \sum_{i,j=1}^3\int_\Omega(\p_{x_j}\rho^{n+1})v_i(\p_{x_i}v_j)dx.
\end{eqnarray*}
Observe that 
\begin{eqnarray*}
&&-\sum_{i,j=1}^3\int_\Omega (\p_{x_j}\rho^{n+1})v_i (\p_{x_i}v_j)dx
=\sum_{i,j=1}^3\int_\Omega \Big\{\rho^{n+1} (\p_{x_j}v_i)( \p_{x_i}v_j)+\rho^{n+1} v_i (\p_{x_j} \p_{x_i}v_j)\Big\}dx\\
&&\quad =\sum_{i,j=1}^3\int_\Omega \rho^{n+1} (\p_{x_j}v_i )(\p_{x_i}v_j) dx+\int_\Omega \rho^{n+1} (v\cdot\nabla) (\nabla\cdot v)dx\\
&&\quad =\sum_{i,j=1}^3\int_\Omega \rho^{n+1}(\p_{x_j}v_i)( \p_{x_i}v_j) dx,\\
&&\sum_{i,j=1}^3\int_\Omega (\p_{x_j}v_i)( \p_{x_i}v_j )dx
=-\sum_{i,j=1}^3\int_\Omega v_i  (\p_{x_j}\p_{x_i}v_j )dx
=-\int_\Omega (v\cdot\nabla) (\nabla\cdot v)dx=0. 
\end{eqnarray*}
Then, we have 
\begin{eqnarray*}
&&I:=-\sum_{i,j=1}^3\int_\Omega (\p_{x_j}\rho^{n+1})v_i (\p_{x_i}v_j)dx=\sum_{i,j=1}^3\int_\Omega \Big(\rho^{n+1}-\frac{M+m}{2}\Big) (\p_{x_j}v_i)( \p_{x_i}v_j) dx,\\
&&|I|\le\max_{x\in\Omega}  \Big|\rho^{n+1}-\frac{M+m}{2}\Big|\norm \nabla v\norm_{L^2(\Omega)^{3\times3}}^2 \le \frac{M-m}{2}\norm \nabla v\norm_{L^2(\Omega)^{3\times3}} ^2.
\end{eqnarray*}
Note that this estimate is given in Chapter 3, Section 4 of \cite{AKM}. Therefore, we obtain
\begin{eqnarray}\label{key3}
\tilde{B}(v,v)&\ge&  \2\int_\Omega \tau^{-1}\rho^{n+1}|v|^2dx+\2\int_\Omega \tau^{-1}\rho^n|v|^2dx\\\nonumber 
&&+ \mu_\ast \norm\nabla v\norm_{L^2(\Omega)^{3\times3}}^2-\theta \frac{M-m}{2}\norm \nabla v\norm_{L^2(\Omega)^{3\times3}} ^2\\\nonumber
&\ge& \min\Big\{ \tau^{-1}m, \mu_\ast-\theta\frac{M-m}{2}  \Big\}\norm v\norm_{H^1(\Omega)^3}^2,\quad \forall\,v\in H^1_{0,\sigma}(\Omega).
\end{eqnarray}
Due to (H4), Lax-Milgram theorem implies that there exists a unique solution  $v^{n+1}\in H^1_{0,\sigma}(\Omega)$ of \eqref{weak31}. 

Finally, we complete the energy estimate for \eqref{weak31}. By \eqref{div3}, \eqref{div4} and \eqref{3key2}, we have 
\begin{eqnarray*}
|g(v^{n+1})|&\le& \tau^{-1}\int_\Omega\rho^n v^n\cdot v^{n+1} dx+\theta \frac{\mu'{}^{\ast}}{m} \norm |\nabla\rho^{n+1}|^2\norm_{L^2(\Omega)}\norm \nabla v^{n+1}\norm_{L^2(\Omega)^{3\times3}}\\
&&+\theta^2\frac{1}{m}\norm |\nabla\rho^{n+1}|^2\norm_{L^2(\Omega)}\norm \nabla v^{n+1}\norm_{L^2(\Omega)^{3\times3}}
+\int_\Omega\rho^{n+1}f^{n+1}\cdot v^{n+1}dx \\
&\le& \2\int_\Omega \tau^{-1}\rho^{n}|v^n|^2dx+\2\int_\Omega \tau^{-1}\rho^n|v^{n+1}|^2dx +\int_\Omega\rho^{n+1}f^{n+1}\cdot v^{n+1}dx \\
&&+\theta\Big( \frac{\mu'{}^{\ast}}{m} +\frac{\theta}{m}\Big)
\tilde{C}_\Omega M \norm\Delta \rho^{n+1}\norm_{L^2(\Omega)}\norm \nabla v^{n+1}\norm_{L^2(\Omega)^{3\times3}}.
\end{eqnarray*}
\eqref{P333} gives 
\begin{eqnarray*}
 &&\norm\Delta \rho^{n+1}\norm_{L^2(\Omega)}\norm \nabla v^{n+1}\norm_{L^2(\Omega)^{3\times3}}\\
&& \le  \Big\{ \frac{\norm \nabla\rho^{n}\norm_{L^2(\Omega)}^2 -\norm \nabla\rho^{n+1}\norm_{L^2(\Omega)}^2}{\theta\tau}     +\frac{\tilde{C}_\Omega^2M^2}{\theta^2} \norm\nabla v^n_\tau\norm_{L^2(\Omega)^{3\times3}}^2 \Big\}^\2 \norm \nabla v^{n+1}\norm_{L^2(\Omega)^{3\times3}}\\
&&=\frac{\tilde{C}_\Omega M}{\theta}
 \Big\{\frac{\theta (\norm \nabla\rho^{n}\norm_{L^2(\Omega)}^2 -\norm \nabla\rho^{n+1}\norm_{L^2(\Omega)}^2)}{\tilde{C}_\Omega^2 M^2\tau}  + \norm\nabla v^n_\tau\norm_{L^2(\Omega)^{3\times3}}^2 \Big\}^\2 \norm \nabla v^{n+1}\norm_{L^2(\Omega)^{3\times3}}\\
 &&\le \frac{1}{2 \tilde{C}_\Omega M\tau}  \Big(\norm \nabla\rho^{n}\norm_{L^2(\Omega)}^2 -\norm \nabla\rho^{n+1}\norm_{L^2(\Omega)}^2\Big)  \\
&&\quad  +  \frac{\tilde{C}_\Omega M}{2\theta}\norm\nabla v^n_\tau\norm_{L^2(\Omega)^{3\times3}}^2   
  + \frac{\tilde{C}_\Omega M}{2\theta}  \norm \nabla v^{n+1}\norm_{L^2(\Omega)^{3\times3}}^2.
\end{eqnarray*}
Hence, \eqref{key3} implies that 
\begin{eqnarray*}
&& \Big\{\Big(\mu_\ast-\theta\frac{M-m}{2} \Big) - \frac{\tilde{C}_\Omega^2 M^2}{2m}(\mu'{}^{\ast} +\theta) \Big\} \norm \nabla v^{n+1}\norm_{L^2(\Omega)^{3\times3}}^2\\
&&\quad -\frac{\tilde{C}_\Omega^2 M^2}{2m}(\mu'{}^{\ast} +\theta)\norm \nabla v^n_\tau\norm_{L^2(\Omega)^{3\times3}}^2\\
&&\quad +  \frac{\tau^{-1}}{2}\norm\sqrt{\rho^{n+1}}v^{n+1}\norm_{L^2(\Omega)^3}^2
 + \frac{\tau^{-1}}{2m}( \mu'{}^{\ast} +\theta)\theta \norm\nabla \rho^{n+1}\norm_{L^2(\Omega)^3}^2 \\
&&\le \frac{\tau^{-1}}{2}\norm\sqrt{\rho^{n}}v^n\norm_{L^2(\Omega)^3}^2 
+\frac{\tau^{-1}}{2m}( \mu'{}^{\ast} +\theta)\theta\norm\nabla \rho^{n}\norm_{L^2(\Omega)^3}^2 \\
&&\qquad + \norm\sqrt{\rho^{n+1}}f^{n+1}\norm_{L^2(\Omega)^3} \norm\sqrt{\rho^{n+1}}v^{n+1}\norm_{L^2(\Omega)^3} .
\end{eqnarray*}
\end{proof}
\begin{Prop}\label{P35}
 The solutions $\rho^{n+1}$, $v^{n+1}$ of \eqref{ellip1}, \eqref{ellip2} satisfy for any $n\ge0$, 
\begin{eqnarray}\label{bbb1}
&&m\norm v^{n+1}\norm_{L^2(\Omega)^3}^2\le \norm\sqrt{\rho^{n+1}}v^{n+1}\norm_{L^2(\Omega)^3}^2\le e^{2n\tau}c_{n} ,\\\label{bbb2}
&&\alpha_2\norm \nabla\rho^{n+1}\norm_{L^2(\Omega)^3}^2\le c_{n}+\sum_{k=0}^n e^{2k\tau}c_{k}\tau,\\ \label{bbb3}
&&2\alpha_1\sum_{k=0}^n\norm \nabla v^{k+1}\norm_{L^2(\Omega)^{3\times3}}^2\tau   
 \le  c_{n}+\sum_{k=0}^n e^{2k\tau}c_{k}\tau,
\end{eqnarray} 
where 
\begin{eqnarray*}
&&\alpha_1:= \mu_\ast-\frac{\theta}{2}(M-m)  - \frac{\tilde{C}_\Omega^2 M^2}{m}(\mu'{}^{\ast} +\theta)  \quad\mbox{($\alpha_1>0$ due to (H3))},\quad 
\alpha_2:=\frac{ \theta}{m}(\mu'{}^{\ast} +\theta),\\
&&c_{n}:= M\norm u\norm_{L^2(\Omega)^3}^2 
 +\alpha_2 \norm\nabla \eta \norm_{L^2(\Omega)^3}^2
  +M\norm f\norm_{L^2([0,n\tau+\tau];L^2(\Omega)^3)}\\
&&\qquad \quad +\frac{\tilde{C}_\Omega^2 M^2}{m}(\mu'{}^{\ast} +\theta)n\tau^2+\alpha_2\tau.
\end{eqnarray*}
\end{Prop}
\begin{proof}
By Proposition \ref{P34}, we have for each $n\ge0$, 
\begin{eqnarray*}
&&\Big\{ \mu_\ast-\frac{\theta}{2}(M-m)  - \frac{\tilde{C}_\Omega^2 M^2}{2m}(\mu'{}^{\ast} +\theta)  \Big\}\sum_{k=0}^n\norm \nabla v^{k+1}\norm_{L^2(\Omega)^{3\times3}}^2\tau  \\
&&\quad - \frac{\tilde{C}_\Omega^2 M^2}{2m}(\mu'{}^{\ast} +\theta)\sum_{k=0}^n\norm \nabla v^{k}_\tau\norm_{L^2(\Omega)^{3\times3}}^2\tau\\
&&\quad +  \2\norm\sqrt{\rho^{n+1}}v^{n+1}\norm_{L^2(\Omega)^3}^2 
 +\frac{\alpha_2}{2} \norm\nabla \rho^{n+1}\norm_{L^2(\Omega)^3}^2 \\
&&\quad\quad\le \2\norm\sqrt{\rho^{0}}v^0\norm_{L^2(\Omega)^3}^2 
+\frac{\alpha_2}{2} \norm\nabla \rho^{0}\norm_{L^2(\Omega)^3}^2 \\
&&\qquad\qquad  + \sum_{k=0}^n\norm\sqrt{\rho^{k+1}}f^{k+1}\norm_{L^2(\Omega)^3} \norm\sqrt{\rho^{k+1}}v^{k+1}\norm_{L^2(\Omega)^3}\tau.
\end{eqnarray*} 
Since $v^0=u$, $v^0_\tau=0$ and $\norm v^{n+1}_\tau-v^{n+1}\norm_{H^1(\Omega)^3}\le\tau$, we have 
\begin{eqnarray}\label{333333bbbbb}
\sum_{k=0}^n\norm \nabla v^{k}_\tau\norm_{L^2(\Omega)^{3\times3}}^2\tau\le \sum_{k=0}^n\norm \nabla v^{k+1}\norm_{L^2(\Omega)^{3\times3}}^2\tau+n\tau^2,\quad n\ge0.
\end{eqnarray}
Hence, noting  that $\rho^0=\eta_\tau$ with $\norm \eta_\tau-\eta\norm_{H^1(\Omega)}\le \tau$, we have 
\begin{eqnarray}\label{35kkkk}
&&\alpha_1\sum_{k=0}^n\norm \nabla v^{k+1}\norm_{L^2(\Omega)^{3\times3}}^2\tau  
+  \2\norm\sqrt{\rho^{n+1}}v^{n+1}\norm_{L^2(\Omega)^3}^2 
 +\frac{\alpha_2}{2} \norm\nabla \rho^{n+1}\norm_{L^2(\Omega)^3}^2 \\\nonumber
&&\le \frac{M}{2}\norm u\norm_{L^2(\Omega)^3}^2 
+\frac{\alpha_2}{2} \norm\nabla \eta \norm_{L^2(\Omega)^3}^2+\frac{\tilde{C}_\Omega^2 M^2}{2m}(\mu'{}^{\ast} +\theta)n\tau^2+\frac{\alpha_2}{2}\tau  \\\nonumber
&&\quad  + \sum_{k=0}^n\norm\sqrt{\rho^{k+1}}f^{k+1}\norm_{L^2(\Omega)^3} \norm\sqrt{\rho^{k+1}}v^{k+1}\norm_{L^2(\Omega)^3}\tau.
\end{eqnarray} 
Set $\frac{\tilde{c}_n}{2}:= \frac{M}{2}\norm u\norm_{L^2(\Omega)^3}^2 
+\frac{\alpha_2}{2} \norm\nabla \eta \norm_{L^2(\Omega)^3}^2+\frac{\tilde{C}_\Omega^2 M^2}{2m}(\mu'{}^{\ast} +\theta)n\tau^2+\frac{\alpha_2}{2}\tau$. Then, we have 
\begin{eqnarray*}
&& \2\norm\sqrt{\rho^{n+1}}v^{n+1}\norm_{L^2(\Omega)^3}^2 \le \frac{\tilde{c}_n}{2}+ \sum_{k=0}^n\norm\sqrt{\rho^{k+1}}f^{k+1}\norm_{L^2(\Omega)^3} \norm\sqrt{\rho^{k+1}}v^{k+1}\norm_{L^2(\Omega)^3}\tau\\
&&\quad\le \frac{\tilde{c}_n}{2}+\frac{M}{2} \sum_{k=0}^n\norm f^{k+1}\norm_{L^2(\Omega)^3}^2\tau+\frac{1}{2} \sum_{k=0}^n\norm\sqrt{\rho^{k+1}}v^{k+1}\norm_{L^2(\Omega)^3}^2\tau, 
\end{eqnarray*}
where 
\begin{eqnarray*}
&&\sum_{k=0}^n\norm f^{k+1}\norm_{L^2(\Omega)^3}^2\tau
=\sum_{k=0}^n\int_\Omega\Big|\int_{k\tau}^{k\tau+\tau} f(t,x)dt  \Big|^2dx\tau^{-1}\\
&&\quad \le\sum_{k=0}^n\int_\Omega\Big\{\Big(\int_{k\tau}^{k\tau+\tau} |f(t,x)|^2dt  \Big)^\2\Big(\int_{k\tau}^{k\tau+\tau}1^2dt  \Big)^\2 \Big\}^2dx\tau^{-1}\\
&&\quad  =\sum_{k=0}^n\int_\Omega\int_{k\tau}^{k\tau+\tau}|f(t,x)|^2dtdx=\norm f\norm_{L^2([0,n\tau+\tau];L^2(\Omega)^3)}^2.
\end{eqnarray*}
Setting $\frac{c_{n}}{2}:=\frac{\tilde{c}_n}{2}+\frac{M}{2}\norm f\norm_{L^2([0,n\tau+\tau];L^2(\Omega)^3)}^2$ and $X^0:=0,\,\,\,X^{n+1}:=(\norm\sqrt{\rho^{1}}v^{1}\norm_{L^2(\Omega)^3}^2+\cdots+\norm\sqrt{\rho^{n+1}}v^{n+1}\norm_{L^2(\Omega)^3}^2)\tau$,  we have 
$$\frac{X^{n+1}-X^n}{\tau}\le c_{n}+X^{n+1},\quad n\ge0.$$
Since $c_0\le c_1\le \cdots$, we have for each $n\ge0$,
$$\frac{X^{k+1}-X^k}{\tau}\le c_{n}+X^{k+1},\quad 0\le k\le n.$$
Hence, as $0<\tau\le\2$, we have for all  $0\le k\le n$,
\begin{eqnarray*}
X^{k+1}+c_n \le (1+2\tau)(X^k+c_n)\le (1+2\tau)^k c_n\le e^{2k\tau}c_n,
\end{eqnarray*}
which leads to 
$$\norm\sqrt{\rho^{n+1}}v^{n+1}\norm_{L^2(\Omega)^3}^2=\frac{X^{n+1}-X^n}{\tau}\le c_{n}+X^{n+1}\le e^{2n\tau}c_n.$$
This estimate and \eqref{35kkkk} yield \eqref{bbb1}-\eqref{bbb3}. 
\end{proof}
\setcounter{section}{3}
\setcounter{equation}{0}
\section{Proof of technical lemmas}

We prove Lemma \ref{key-lemma1} and Lemma \ref{key-lemma2}. 

 \begin{proof}[{\bf Proof of  Lemma \ref{key-lemma1}.}] The proof is similar to that of Lemma 4.3 in Soga \cite{Soga2023}. 
Note that (A2) implies that the value $\rho_k(t,\cdot)$ is determined for every $t\in[0,T]$ (one cannot change the value even on a null set of $[0,T]$). 
We use an Ascoli-Arzela type reasoning. Set $\{s_k\}_{k\in\N}:=\Q\cap[0,T]$. Since $\{\rho_k(s_1,\cdot)\}_{k\in \N}$ is bounded in $L^2(\Omega)$, there exists a subsequence   $\{\rho_{1l}\}_{l\in\N}\subset \{\rho_k\}_{k\in\N}$ and $\rho(s_1,\cdot)\in L^2(\Omega)$ such that $\rho_{1l}(s_1,\cdot)\wto \rho(s_1,\cdot)$ in $L^2(\Omega)$ as $l\to\infty$.  
 It holds that $\alpha\le\rho(s_1,\cdot)\le \beta$. 
 In fact,  set $\tilde{\rho}(x):=\min\{ \rho(s_1,x)-\alpha,0 \}:  \Omega\to\R_{\le0}$; 
 since $\rho_{1l}(s_1,\cdot)-\alpha\ge0$ a.e. by assumption, we have $(\rho_{1l}(s_1,\cdot)-\alpha, \tilde{\rho})_{L^2(\Omega)}\le0$ for all $l$ and  $(\rho_{1l}(s_1,\cdot)-\alpha, \tilde{\rho})_{L^2(\Omega)}\to (\rho(s_1,\cdot)-\alpha, \tilde{\rho})_{L^2(\Omega)}=\norm\tilde{\rho}\norm_{L^2(\Omega)}^2$ as $l\to\infty$; hence $\norm\tilde{\rho}\norm_{L^2(\Omega)}^2\le0$ and $\tilde{\rho}=0$, i.e., $\rho(s_1,\cdot)\ge\alpha$; 
similarly, set $\tilde{\rho}(x):=\min\{ \beta-\rho(s_1,\cdot),0 \}:  \Omega\to\R_{\le0}$; since $\beta-\rho_{1l}(s_1,\cdot)\ge0$ a.e. by assumption, we have $(\beta-\rho_{1l}(s_1,\cdot), \tilde{\rho})_{L^2(\Omega)}\le0$ for all $l$ and  $(\beta-\rho_{1l}(s_1,\cdot), \tilde{\rho})_{L^2(\Omega)}=(\beta, \tilde{\rho})_{L^2(\Omega)}-(\rho_{1l}(s_1,\cdot), \tilde{\rho})_{L^2(\Omega)}\to (\beta-\rho(s_1,\cdot), \tilde{\rho})_{L^2(\Omega)}=\norm\tilde{\rho}\norm_{L^2(\Omega)}^2$ as $l\to\infty$; hence $\norm\tilde{\rho}\norm_{L^2(\Omega)}^2\le0$ and $\tilde{\rho}=0$, i.e., $\rho(s_1,\cdot)\le\beta$. 

Since $\{\rho_{1l}(s_2,\cdot)\}_{l\in\N}$ is bounded in $L^2(\Omega)$, there exists a subsequence   $\{\rho_{2l}\}_{l\in\N}\subset \{\rho_{1l}\}_{l\in\N}$ and $\rho(s_2,\cdot)\in L^2(\Omega)$ such that $\rho_{2l}(s_2,\cdot)\wto \rho(s_2,\cdot)$ in $L^2(\Omega)$ as $l\to\infty$, where $\alpha\le\rho(s_2,\cdot)\le \beta$. Repeating this process, we obtain a subsequence   $\{\rho_{k+1l}\}_{l\in\N}\subset \{\rho_{kl}\}_{l\in\N}$ and $\rho(s_{k+1},\cdot)\in L^2(\Omega)$ such that $\rho_{k+1l}(s_{k+1},\cdot)\wto \rho(s_{k+1},\cdot)$ in $L^2(\Omega)$ as $l\to\infty$ with $\alpha\le\rho(s_{k+1},\cdot)\le \beta$, for each $k\in\N$.  It is clear that the sequence $\{ \rho_{kk}  \}_{k\in\N}\subset \{\rho_k\}_{k\in\N}$ satisfies 
$$\rho_{kk}(s_{k'},\cdot)\wto \rho(s_{k'},\cdot)\mbox{\quad in $L^2(\Omega)$ as $k\to\infty$, $\forall\,k'\in\N$}.$$

Hereafter, we re-write $\{ \rho_{kk}  \}_{k\in\N}$ as   $\{\rho_k\}_{k\in\N}$.  
In order to see weak convergence of $\{\rho_k(t,\cdot)\}_{k\in\N}$ for all $t\in[0,T]$, we use weak equi-continuity of $(\rho_k,\phi)_{L^2(\Omega)}$ for each fixed $\phi\in C^\infty_0(\Omega)$, i.e., for any $\ep>0$, there exists $\delta=\delta(\ep,\phi)>0$ such that  
$$|\tilde{t}-t|<\delta\Rightarrow |(\rho_k(\tilde{t},\cdot),\phi)_{L^2(\Omega)}-(\rho_k(t,\cdot),\phi)_{L^2(\Omega)}|<\frac{\ep}{3},\quad\forall\,k\in\N.$$
For $\delta=\delta(\ep,\phi)$, introduce $I_0:=[0,\delta],I_1:=[\delta,2\delta],\ldots,I_{J}:=[J\delta,T]$, where $J=J(\ep,\phi)$. Take a rational number $\tilde{s}_j$ from the interior of each $I_j$, $0\le j\le J$ ($0\le j\le J-1$ if $J\delta=T$). For any $t\in [0,T]$, we find $I_j$ such that $t\in I_j$, where $|t-\tilde{s}_j|<\delta$. 
Since $\{(\rho_k(\tilde{s}_j,\cdot),\phi)_{L^2(\Omega)}\}_{k\in\N}$ is a convergent sequence of $\R$, there exists $K_j\in \N$ such that if $k,k'\ge K_j$ we have 
 $$|(\rho_{k'}(\tilde{s}_j,\cdot),\phi)_{L^2(\Omega)}-(\rho_k(\tilde{s}_j,\cdot),\phi)_{L^2(\Omega)}|< \frac{\ep}{3}.$$
 Set $K:=\max\{K_0,K_1\ldots,K_J\}.$
Then, we  have for any $k,k'\ge K$, 
\begin{eqnarray*}
&&|(\rho_{k'}(t,\cdot),\phi)_{L^2(\Omega)}-(\rho_k(t,\cdot),\phi)_{L^2(\Omega)}|
\le |(\rho_{k'}(t,\cdot),\phi)_{L^2(\Omega)}-(\rho_{k'}(\tilde{s}_j,\cdot),\phi)_{L^2(\Omega)}|\\
&&\quad +|(\rho_{k'}(\tilde{s}_j,\cdot),\phi)_{L^2(\Omega)}-(\rho_k(\tilde{s}_j,\cdot),\phi)_{L^2(\Omega)}|
+ |(\rho_{k}(\tilde{s}_j,\cdot),\phi)_{L^2(\Omega)}-(\rho_k(t,\cdot),\phi)_{L^2(\Omega)}|\\
&&\quad< \frac{\ep}{3}+\frac{\ep}{3}+\frac{\ep}{3}
=\ep.
\end{eqnarray*} 
Therefore, $\{(\rho_k(t,\cdot),\phi)_{L^2(\Omega)}\}_{k\in\N}$ is a convergent sequence of $\R$ for any $\phi\in C^\infty_0(\Omega)$.
 
On the other hand, since $\{\rho_k(t,\cdot)\}_{k\in\N}$ is bounded in $L^2(\Omega)$, we have a subsequence $\{\tilde{\rho}_k(t,\cdot)\}_{k\in\N}\subset \{\rho_k(t,\cdot)\}_{k\in\N}$ and $\rho(t,\cdot)\in L^2(\Omega)$ such that  $\alpha\le\rho(t,\cdot)\le \beta$ and 
$$\tilde{\rho}_k(t,\cdot)\wto \rho(t,\cdot)\mbox{\quad in $L^2(\Omega)$ as $k\to\infty$},$$
which implies that 
$$\lim_{k\to\infty} (\rho_k(t,\cdot),\phi)_{L^2(\Omega)}=\lim_{k\to\infty} (\tilde{\rho}_k(t,\cdot),\phi)_{L^2(\Omega)} =(\rho(t,\cdot),\phi)_{L^2(\Omega)},\quad\forall\,\phi\in C^\infty_0(\Omega).$$
Since $C^\infty_0(\Omega)$ is dense in $L^2(\Omega)$,  we conclude that   $\rho_k(t,\cdot)\wto \rho(t,\cdot)$ in $L^2(\Omega)$ as $k\to\infty$ for every $t\in [0,T]$. 
\end{proof}
\begin{proof}[{\bf Proof of  Lemma \ref{key-lemma2}.}]
The proof is similar to that of Lemma 4.4 in \cite{Soga2023}.  
First we find $A_\lambda$ for each fixed $t\in[0,T]$.  
Suppose that the assertion does not hold. 
Then, there exists some constant $\lambda_0>0$ for which there exist $k(i), l(i)\in\N$ for each $i\in\N$  such that 
\begin{eqnarray}\label{6161}
&&\norm v_{k(i)}(t,\cdot)-v_{l(i)}(t,\cdot)\norm_{L^2(\Omega)^3}\\\nonumber 
&&\quad > \lambda_0 (\norm v_{k(i)}(t,\cdot)\norm_{H^1(\Omega)^3} +\norm v_{l(i)}(t,\cdot)\norm_{H^1(\Omega)^3}+k(i)^{-1}+l(i)^{-1}) \\\nonumber
&&\qquad+i\Big(\sup_{\varphi\in S}\Big|\Big( \rho_{k(i)}(t,\cdot)v_{k(i)}(t,\cdot)-\rho_{l(i)}(t,\cdot)v_{l(i)}(t,\cdot),\varphi \Big)_{L^2(\Omega)^3}\Big|+k(i)^{-1}+l(i)^{-1}\Big). 
\end{eqnarray}
Due to the presence of $k(i)^{-1}, l(i)^{-1}$, 
\begin{itemize}
\item[($\ast$)] {\it At least one of $\{k(i)\}_{i\in\N}, \{l(i)\}_{i\in\N}$ must be unbounded; if $\{l(i)\}_{i\in\N}$ (resp. $\{k(i)\}_{i\in\N}$) is bounded, $\{k(i)\}_{i\in\N}$ (resp. $\{l(i)\}_{i\in\N}$) is unbounded and $\norm v_{k(i)}(t,\cdot)\norm_{L^2(\Omega)^3}\to\infty$ (resp. $\norm v_{l(i)}(t,\cdot)\norm_{L^2(\Omega)^3}\to\infty$) as $i\to\infty$. }
\end{itemize}

Normalize $v_{k(i)}(t,\cdot),v_{l(i)}(t,\cdot)$ as   
\begin{eqnarray}\label{ooo}
&&\omega_{i}^1:=\frac{v_{k(i)}(t,\cdot)}{\norm v_{k(i)}(t,\cdot)\norm_{H^1(\Omega)^3}+\norm v_{l(i)}(t,\cdot)\norm_{H^1(\Omega)^3}+k(i)^{-1}+l(i)^{-1}},\\\label{oooo}
&&\omega_{i}^2:=\frac{v_{l(i)}(t,\cdot)}{\norm v_{k(i)}(t,\cdot)\norm_{H^1(\Omega)^3}+\norm v_{l(i)}(t,\cdot)\norm_{H^1(\Omega)^3}+k(i)^{-1}+l(i)^{-1}}.  
\end{eqnarray}
Since $\{\omega^1_{i}\}_{i\in\N}$, $\{\omega^2_{i}\}_{i\in\N}$ are bounded sequences of $H^1_0(\Omega)^3$, it follows from  the Rellich-Kondrachov theorem that there exist $\omega^1,\omega^2\in L^2(\Omega)^3$ such that
\begin{eqnarray}\label{6015}
\quad \mbox{$\omega^1_{i}\to \omega^1$,\,\, $\omega^2_{i}\to \omega^2$ strongly in $L^2(\Omega)^3$ as $i\to\infty$ (up to subsequences).}
 \end{eqnarray}
We show that $\omega^1,\omega^2\in H^1_{0,\sigma}(\Omega)$. For $j=1,2,3$, since $\{\p_{x_j}\omega^1_{i}\}_{i\in\N}$ is a bounded sequence of $L^2(\Omega)^3$, we have $w^1_j\in L^2(\Omega)^3$ such that for any $\phi\in C^\infty_{0}(\Omega)$,
\begin{eqnarray*}
&&\int_\Omega \p_{x_j}\omega^1_{i} \phi dx=-\int_\Omega\omega^1_{i}  \p_{x_j}\phi dx\to  
\int_\Omega w^1_j\phi dx
=-\int_\Omega\omega^1 \p_{x_j}\phi dx\\
&& \mbox{as $i\to\infty$ (up to a subsequence)},
\end{eqnarray*}
which implies that $\omega^1\in H^1(\Omega)^3$ and $w^1_j=\p_{x_j}\omega^1$. In particular, $\{\omega^1_{i}\}_{i\in\N}$ weakly converges to $\omega^1$ in $H^1(\Omega)^3$ as $i\to\infty$ (up to a subsequence), i.e.,  
$$(\omega^1_i,\varphi)_{H^1(\Omega)^3}\to(\omega^1,\varphi)_{H^1(\Omega)^3}\mbox{ as $i\to\infty$},\quad\forall\,\varphi\in H^1(\Omega)^3.$$
Since $\{\omega^1_{i}\}_{i\in\N}$ is a bounded sequence of  $H^1_{0,\sigma}(\Omega)$, we have $\tilde{\omega}^1\in H^1_{0,\sigma}(\Omega)$ such that $\omega_i\wto\tilde{\omega}^1$ in $H^1_{0,\sigma}(\Omega)$ as $i\to\infty$ (up to a subsequence), i.e., 
$$(\omega^1_i,\varphi)_{H^1(\Omega)^3}\to(\tilde{\omega}^1,\varphi)_{H^1(\Omega)^3}\mbox{ as $i\to\infty$},\quad\forall\,\varphi\in H^1_{0,\sigma}(\Omega).$$
Hence, we have 
\begin{eqnarray*}
&&0=(\omega^1_i, \varphi)_{H^1(\Omega)^3}-(\omega^1_i, \varphi)_{H^1(\Omega)^3}\to 0=(\omega^1-\tilde{\omega}^1, \varphi)_{H^1(\Omega)^3}\mbox{ as $i\to\infty$},\mbox{\quad $\forall\, \varphi\in H^1_{0,\sigma}(\Omega)$}.
\end{eqnarray*}
Therefore, noting that $\omega^1_i-\tilde{\omega}^1\in H^1_{0,\sigma}(\Omega)$ and $\omega^1-\tilde{\omega}^1\in H^1(\Omega)^3$, we have 
\begin{eqnarray*}
0&=&(\omega^1-\tilde{\omega}^1,  \omega^1_i-\tilde{\omega}^1)_{H^1(\Omega)^3}
=(\omega^1-\tilde{\omega}^1,  \omega^1-\tilde{\omega}^1)_{H^1(\Omega)^3}
+ (\omega^1-\tilde{\omega}^1,  \omega^1_i-\omega^1)_{H^1(\Omega)^3}\\
&&\to 0=\norm\omega^1-\tilde{\omega}^1\norm_{H^1(\Omega)^3}^2
\mbox{\quad as $i\to\infty$.}
\end{eqnarray*}
This implies that $\omega^1=\tilde{\omega}^1\in  H^1_{0,\sigma}(\Omega)$. The same reasoning yields  $\omega^2 \in H^1_{0,\sigma}(\Omega)$.
 
It follows from \eqref{6161} that 
\begin{eqnarray}\label{key1}
2&\ge&\norm \omega^1_{i}- \omega^2_{i} \norm_{L^2(\Omega)^3}> \lambda_0 +i\Big(\sup_{\varphi\in S}\Big|\Big( \rho_{k(i)}(t,\cdot)\omega^1_i-\rho_{l(i)}(t,\cdot)\omega^2_i,\varphi \Big)_{L^2(\Omega)^3}\Big|\Big)\\\nonumber 
&& +i\frac{k(i)^{-1}+l(i)^{-1}}{\norm v_{k(i)}(t,\cdot)\norm_{H^1(\Omega)^3} +\norm v_{l(i)}(t,\cdot)\norm_{H^1(\Omega)^3}+k(i)^{-1}+l(i)^{-1}}\\\nonumber
&\ge&\lambda_0>0,\mbox{ $\forall\,i\in\N$},
\end{eqnarray}
which implies that 
\begin{eqnarray}\label{key2}
&&\sup_{\varphi\in S}\Big|\Big( \rho_{k(i)}(t,\cdot)\omega^1_i-\rho_{l(i)}(t,\cdot)\omega^2_i,\varphi \Big)_{L^2(\Omega)^3}\Big|\to 0\mbox{\quad as $i\to\infty$.}
\end{eqnarray} 
 For each $\varphi\in C^\infty_{0,\sigma}(\Omega)$ with  $\norm \varphi \norm_{W^{3,\infty}(\Omega)^3} =1$, we obtain if  $\{k(i)\}_{i\in\N},\{l(i)\}_{i\in\N}$ are both unbounded,  
\begin{eqnarray}\label{unun}
&&\sup_{\varphi\in S}\Big|\Big( \rho_{k(i)}(t,\cdot)\omega^1_i-\rho_{l(i)}(t,\cdot)\omega^2_i,\varphi \Big)_{L^2(\Omega)^3}\Big|\\\nonumber
&&\quad \ge \Big|( \rho_{k(i)}(t,\cdot) \omega^1_{i},  \varphi)_{L^2(\Omega)^3}
-(\rho_{l(i)}(t,\cdot) \omega^2_{i},  \varphi)_{L^2(\Omega)^3}\Big|\\\nonumber
&&\quad = \Big|( \rho_{k(i)}(t,\cdot),  \omega^1\cdot  \varphi)_{L^2(\Omega)}
+( \rho_{k(i)}(t,\cdot), (\omega^1_{i}- \omega^1)  \cdot\varphi)_{L^2(\Omega)}\\\nonumber
&&\qquad -(\rho_{l(i)}(t,\cdot) , \omega^2 \cdot\varphi)_{L^2(\Omega)}
-(\rho_{l(i)}(t,\cdot), (\omega^2_{i}-\omega^2)  \cdot\varphi)_{L^2(\Omega)}
\Big|\\\nonumber
&&\qquad \to\Big|    (\rho(t,\cdot),\omega^1\cdot\varphi)_{L^2(\Omega)}-  (\rho(t,\cdot),\omega^2\cdot\varphi)_{L^2(\Omega)} \Big|\mbox{ as $i\to\infty$ (due to (A3))};
\end{eqnarray}
 otherwise, due to the statement $(\ast)$, either $\omega^1=0$ or $\omega^2=0$ and \eqref{unun} still holds. 
Hence, setting $\omega:=\omega^1-\omega^2$, we obtain with \eqref{6015},  (\ref{key1}) and (\ref{key2}),       
$$ 0<\lambda_0\le\norm \omega\norm_{L^2(\Omega)^3},\quad(\rho(t,\cdot)\omega,\varphi)_{L^2(\Omega)^3}=0,\mbox{ \,\,\,$\forall\,\varphi\in  C^\infty_{0,\sigma}(\Omega)$.}$$
The first inequality implies $\omega\neq0$. However,  since $\omega\in H^1_{0,\sigma}(\Omega)$, we take $\{\omega_l\}_{l\in\N}\subset C^\infty_{0,\sigma}(\Omega)$ that approximates $\omega$ in the $H^1(\Omega)^3$-norm as $l\to\infty$ and find    
\begin{eqnarray*}
\int_\Omega \rho(t,x)|\omega(x)|^2dx&=&(\rho(t,\cdot)\omega,\omega)_{L^2(\Omega)^3}=(\rho(t,\cdot)\omega,\omega_l)_{L^2(\Omega)^3}+(\rho(t,\cdot)\omega,\omega-\omega_l)_{L^2(\Omega)^3}\\
&=&(\rho(t,\cdot)\omega,\omega-\omega_l)_{L^2(\Omega)^3}\to0\mbox{\quad as $l\to\infty$}.
\end{eqnarray*}
Since $0<\alpha\le\rho(t,\cdot)\le \beta$ by assumption (A1), we conclude $\omega=0$, which is  a contradiction.  Therefore, there exists $A_\lambda=A_\lambda(t)\ge0$ for each $t\in[0,T]$. 

We show that  there exists  $A_\lambda\ge0$ independent of the choice of  $t\in[0,T]$.  Fix any $\lambda>0$. 
Let $A^\ast_\lambda(t)$ be the infimum of $\{A_\lambda\,|\, \mbox{\eqref{key5555} holds} \}$ for each fixed $t$. We will see  that $A^\ast_\lambda(\cdot)$ is bounded on $[0,T]$. 
Suppose that  $A^\ast_\lambda(\cdot)$ is not bounded. Then, we find a sequence $\{s_i\}_{i\in\N}\subset[0,T]$ for which $A_\lambda^\ast(s_i)\nearrow\infty$ as $i\to\infty$.  Set $A_i:=A^\ast_\lambda(s_i)/2$. For each $i\in \N$, there exists $k(i),l(i)$ for which we have   
\begin{eqnarray*}
 &&\norm v_{k(i)}(s_i,\cdot)-v_{l(i)}(s_i,\cdot)\norm_{L^2(\Omega)^3}\\
 &&\quad > \lambda (\norm v_{k(i)}(s_i,\cdot)\norm_{H^1(\Omega)^3}+\norm v_{l(i)}(s_i,\cdot)\norm_{H^1(\Omega)^3}+k(i)^{-1}+l(i)^{-1})\\
&&\qquad +A_i \Big(\sup_{\varphi}\Big|\Big( \rho_{k(i)}(s_i,\cdot)v_{k(i)}(s_i,\cdot)-\rho_{l(i)}(s_i,\cdot)v_{l(i)}(s_i,\cdot),\varphi \Big)_{L^2(\Omega)^3}\Big|+k(i)^{-1}+l(i)^{-1}\Big). 
\end{eqnarray*}
Note that $A_i\nearrow\infty$ as $i\to\infty$ and $\{s_i\}_{i\in\N}$ converges to some $t^\ast\in[0,T]$ as $i\to\infty$ (up to a subsequence). Then, we may follow the same reasoning as the first half of our proof  and reach a contradiction. In fact, we obtain the limit functions $\omega^1$, $\omega^2$ of $\{\omega^1_i\}_{i\in\N}$, $\{\omega^2_i\}_{i\in\N}$ defined by \eqref{ooo}, \eqref{oooo} with $s_i$ in place of $t$, where we note that $\{\omega^1_i\}_{i\in\N}$, $\{\omega^2_i\}_{i\in\N}$ are still sequences of $H^1_{0,\sigma}(\Omega)$; $ \omega=\omega^1-\omega^2$ satisfies  $0<\lambda\le\norm\omega\norm_{L^2(\Omega)^3}$ and  $(\rho(t^\ast,\cdot)\omega,\varphi)_{L^2(\Omega)^3}=0$ for all $\varphi\in  C^\infty_{0,\sigma}(\Omega)$ as  
 \begin{eqnarray*}
 &&\sup_{\varphi\in S}\Big|\Big( \rho_{k(i)}(s_i,\cdot)\omega^1_i-\rho_{l(i)}(s_i,\cdot)\omega^2_i,\varphi \Big)_{L^2(\Omega)^3}\Big|\to 0\mbox{ as $i\to\infty$},\\
&&\sup_{\varphi\in S}\Big|\Big( \rho_{k(i)}(s_i,\cdot)\omega^1_i-\rho_{l(i)}(s_i,\cdot)\omega^2_i,\varphi \Big)_{L^2(\Omega)^3}\Big|\\
&&\quad \ge \Big|( \rho_{k(i)}(s_i,\cdot) \omega^1_{i},  \varphi)_{L^2(\Omega)^3}
-(\rho_{l(i)}(s_i,\cdot) \omega^2_{i},  \varphi)_{L^2(\Omega)^3}\Big|\\
&&\quad = \Big|\Big( \rho_{k(i)}(t^\ast,\cdot),  \omega^1\cdot  \varphi\Big)_{L^2(\Omega)}
+\Big( \rho_{k(i)}(s_i,\cdot)- \rho_{k(i)}(t^\ast,\cdot),  \omega^1\cdot\varphi\Big)_{L^2(\Omega)}\\
&&\qquad +\Big(\rho_{k(i)}(s_i,\cdot), (\omega_i^1-\omega^1)\cdot \varphi\Big)_{L^2(\Omega)}
-\Big( \rho_{l(i)}(t^\ast,\cdot),  \omega^2 \cdot \varphi\Big)_{L^2(\Omega)}\\
&&\qquad -\Big( \rho_{l(i)}(s_i,\cdot)- \rho_{l(i)}(t^\ast,\cdot),  \omega^2\cdot\varphi\Big)_{L^2(\Omega)}
 -\Big(\rho_{l(i)}(s_i,\cdot), (\omega_i^2-\omega^2\cdot) \varphi\Big)_{L^2(\Omega)}
\Big|\\
&& \to\Big|    (\rho(t^\ast,\cdot),\omega^1\cdot\varphi)_{L^2(\Omega)}-  (\rho(t^\ast,\cdot),\omega^2\cdot\varphi)_{L^2(\Omega)} \Big|\mbox{ as $i\to\infty$ (due to (A3))},
\end{eqnarray*}
where we use  the statement $(\ast)$ and weak equi-continuity of $\{\rho_k\}_{k\in\N}$ with smooth approximation of $\omega^1$ and $\omega^2$, i.e., 
for any $\ep>0$ take $\omega_\ep\in C^\infty_0(\Omega)$  such that $\norm \omega_\ep-\omega^1\norm_{L^2(\Omega)^3}<\ep$ and observe 
\begin{eqnarray*}
&&\Big|\Big( \rho_{k(i)}(s_i,\cdot)- \rho_{k(i)}(t^\ast,\cdot),  \omega^1\cdot\varphi\Big)_{L^2(\Omega)}\Big|\le \Big|\Big( \rho_{k(i)}(s_i,\cdot)- \rho_{k(i)}(t^\ast,\cdot),  \omega_\ep\cdot\varphi\Big)_{L^2(\Omega)}\Big|\\
&&\qquad +\Big|\Big( \rho_{k(i)}(s_i,\cdot)- \rho_{k(i)}(t^\ast,\cdot),  (\omega^1-\omega_\ep)\cdot\varphi\Big)_{L^2(\Omega)}\Big|\\
&&\quad \le  \Big|\Big( \rho_{k(i)}(s_i,\cdot)- \rho_{k(i)}(t^\ast,\cdot),  \omega_\ep\cdot\varphi\Big)_{L^2(\Omega)}\Big|+2\beta \ep \to 2\beta \ep\mbox{ as $i\to\infty$}.
\end{eqnarray*}
Thus, we reach a contradiction and complete the proof.
\end{proof}
\setcounter{section}{4}
\setcounter{equation}{0}
\section{Convergence}
For each $\tau>0$, we interpolate the solution of the time-discrete problem as $\rho_\tau,\tilde{\rho}_\tau:[0,\infty)\times\Omega\to\R$, $v_\tau:[0,\infty)\times\Omega\to\R^3$,  
\begin{eqnarray*}
&&\rho_\tau(t,x):=\rho^{n+1}(x)\mbox{\quad for $t\in(\tau n,\tau n+\tau]$},\quad \rho_\tau(0,x):=\rho_\tau(\tau,x),\\ 
&&\tilde{\rho}_\tau(t,x):=\rho^n(x)+\frac{\rho^{n+1}(x)-\rho^{n}(x)}{\tau}(t-\tau n)\mbox{\quad for $t\in[\tau n,\tau n+\tau]$},\\
&&v_\tau(t,x):=v^{n+1}(x) \mbox{\quad for $t\in(\tau n,\tau n+\tau]$},\quad v_\tau(0,x):=v_\tau(\tau,x). 
\end{eqnarray*}
Let $\{\tau_k\}_{k\in\N}$ be a sequence such that $\tau_k\to0^+$ as $k\to\infty$. We re-write $\rho_{\tau_k}, \tilde{\rho}_{\tau_k},v_{\tau_k}$ as $\rho_k,\tilde{\rho}_k, v_k$. 
Let $T>0$ be an arbitrary number. 
Restricting $(t,x)$ to $[0,T]\times\Omega$, we investigate convergence (up to a subsequence) of $\{\rho_{k}, v_{k}\}_{k\in\N}$  in order to obtain a weak $[0,T]$-solution of \eqref{NS1}, where $\rho_{k}|_{[0,T]\times\Omega}, \tilde{\rho}_{k}|_{[0,T]\times\Omega},v_{k}|_{[0,T]\times\Omega} $ are still denoted by $\rho_{k}, \tilde{\rho}_{k},v_{k}$; then, we discuss existence of  a global weak solution.  
 
Here are properties of $\{ \rho_k\}_{k\in\N},\{ \tilde{\rho}_k\}_{k\in\N}, \{ v_k\}_{k\in\N}$ instantly seen from Section 3: 
{\it \begin{itemize}
\item Proposition \ref{P32} implies that $m\le \rho_k\le M$, $m\le \tilde{\rho}_k\le M$ for all $k$ and $\int_\Omega \rho_k(t,x)dx=\int_\Omega \tilde{\rho}_k(t,x)dx= \int_\Omega \eta_{\tau_k}(x)dx $ for all $t\in[0,T]$ and $k$.
\item \eqref{bbb2} implies that there exists a constant $C_1(T)$ such that $\norm \rho_k\norm_{L^\infty([0,T]; H^1(\Omega))}\le C_1(T)$ for all $k$.
\item Proposition \ref{P31}, \eqref{P333}, \eqref{333333bbbbb} and  \eqref{bbb3} imply that  $\rho_k\in L^2([0,T]; H^2_N(\Omega))$ for all $k$ and   there exists a constant $C_2(T)$ such that $\norm \rho_k\norm_{L^2([0,T]; H^2(\Omega))}\le C_2(T)$ for all $k$.  
\item \eqref{ellip1}, \eqref{P333} and \eqref{bbb2} imply that  there exists a constant $C_3(T)$ such that 
$\norm\rho_k-\tilde{\rho}_k\norm_{L^2([0,T];L^2(\Omega))}
\le C_3(T)\tau_k $ for all $k$.
\item Proposition \ref{P34} and Proposition \ref{P35} with \eqref{333333bbbbb} imply that $v_k\in L^\infty([0,T];L^2(\Omega)^3)\cap L^2([0,T];H^1_{0,\sigma}(\Omega))$ for all $k$ and  there exist  constants $C_4(T),C_5(T)$ such that \\$\norm v_k\norm_{L^\infty([0,T]; L^2(\Omega)^3)}\le C_4(T)$ and $\norm v_k\norm_{L^2([0,T]; H^1(\Omega)^3)}\le C_5(T)$ for all $k$. 
\end{itemize}}
\begin{Prop}\label{weak-convergence}
There exists a subsequence of  $\{ \rho_k,  v_k\}_{k\in\N}$, still denoted by the same symbol, and functions $$\mbox{$\rho\in L^2([0,T];H^2_N(\Omega))$, $v\in L^2([0,T];H^1_{0,\sigma}(\Omega))$, $V\in L^2([0,T];L^2(\Omega)^3)$}$$
 for which  the following weak convergence holds: 
\begin{eqnarray}\nonumber
&&\rho_k \wto \rho \mbox{ \quad \quad \quad \quad \quad \quad in $L^2([0,T];L^2(\Omega))$ as $k\to\infty$},\\ \nonumber
&&\p_{x_i}\rho_k \wto \p_{x_i}\rho \mbox{\quad\quad\quad\quad\!in $L^2([0,T];L^2(\Omega))$ as $k\to\infty$ ($i=1,2,3$)},\\\nonumber
&&\p_{x_i}\p_{x_j}\rho_k \wto \p_{x_i}\p_{x_j}\rho \mbox{ \quad in $L^2([0,T];L^2(\Omega))$ as $k\to\infty$ ($i,j=1,2,3$)},\\\nonumber
&&v_k \wto v \mbox{\quad\quad \quad\quad\quad\quad\,\! in $L^2([0,T];L^2(\Omega)^3)$ as $k\to\infty$},\\\nonumber
&&\p_{x_i}v_k \wto \partial_{x_i} v \mbox{\quad\quad\quad\quad in $L^2([0,T];L^2(\Omega)^3)$ as $k\to\infty$ ($i=1,2,3$)},\\\nonumber
&&\rho_k v_k \wto V\mbox{\quad\quad\quad\quad\quad\,in $L^2([0,T];L^2(\Omega)^3)$ as $k\to\infty$}.
\end{eqnarray} 
\end{Prop}
\begin{proof}
Since  $\{\rho_k\}_{k\in\N}$, $\{\p_{x_i}\rho_k\}_{k\in\N}$, $\{\p_{x_i}\p_{x_j}\rho_k\}_{k\in\N}$ are bounded in $L^2([0,T];L^2(\Omega))$, there is a subsequence $\{a_1(k)\}_{k\in\N}\subset\N$ and $\rho, r^i, r^{ij}\in L^2([0,T];L^2(\Omega))$ such that $\rho_{a_1(k)}\wto \rho$, $\p_{x_i}\rho_{a_1(k)}\wto r^i$, $\p_{x_i}\p_{x_j}\rho_{a_1(k)}\wto r^{ij}$ in $L^2([0,T];L^2(\Omega))$ as $k\to\infty$. This implies that  $\int_0^t\int_\Omega \rho_{a_1(k)}\p_{x_i}\phi dxdt=-\int_0^t\int_\Omega \p_{x_i}\rho_{a_1(k)}\phi dxdt\to \int_0^t\int_\Omega \rho\p_{x_i}\phi dxdt=-\int_0^t\int_\Omega r^i\phi dxdt$ as $k\to\infty$ for all 
 $\phi\in C^\infty([0,T]\times\Omega;\R)$ with supp$(\phi)\subset(0,T)\times\Omega$. Hence, we have $r^i=\p_{x_i}\rho$. Similarly observation shows $r^{ij}=\p_{x_i}\p_{x_j}\rho$, as well as $\rho_{a_1(k)}\wto \rho$ in $L^2([0,T];H^2(\Omega))$ as $k\to\infty$.   On the other hand,   $\{\rho_{a_1(k)}\}_{k\in\N}$ is bounded in the Hilbert space $L^2([0,T];H^2_N(\Omega))$, there is a subsequence $\{a_2(k)\}_{k\in\N}\subset \{a_1(k)\}_{k\in\N}$ and $\bar{\rho}\in L^2([0,T];H^2_N(\Omega))$ such that $\rho_{a_2(k)}\wto \bar{\rho}$ in $L^2([0,T];H^2_N(\Omega))$ as $k\to\infty$. Therefore, for any $\phi\in L^2([0,T];H^2_N(\Omega))$, it holds that $0=  (\rho_{a_2(k)},\phi)_{L^2([0,T];H^2(\Omega))}- (\rho_{a_2(k)},\phi)_{L^2([0,T];H^2(\Omega))}\to (\rho-\bar{\rho},\phi)_{L^2([0,T];H^2(\Omega))}=0$ as $k\to\infty$. Since $\rho_{a_2(k)}-\bar{\rho}\in L^2([0,T];H^2_N(\Omega))$, we have 
 \begin{eqnarray*}
0&=&(\rho-\bar{\rho},\rho_{a_2(k)}-\bar{\rho})_{L^2([0,T];H^2(\Omega))}\\
&=&(\rho-\bar{\rho},\rho-\bar{\rho})_{L^2([0,T];H^2(\Omega))}+(\rho-\bar{\rho},\rho_{a_2(k)}-\rho)_{L^2([0,T];H^2(\Omega))}\\
&\to&\norm \rho-\bar{\rho}\norm_{L^2([0,T];H^2(\Omega))}\mbox{\quad as $k\to\infty$,}
\end{eqnarray*}
 which means $\rho=\bar{\rho}\in L^2([0,T];H^2_N(\Omega))$.  

Since $\{v_{a_2(k)}\}_{k\in\N}$, $\{\p_{x_i}v_{a_2(k)}\}_{k\in\N}$ are bounded in $L^2([0,T];L^2(\Omega)^3)$, there is a subsequences $\{ a_3(k) \}_{k\in\N}\subset \{a_2(k)\}_{k\in\N}$ and  $v, w^i\in L^2([0,T];L^2(\Omega)^3)$ such that $v_{a_3(k)}\weakto v$ and $\p_{x_i}v_{a_3(k)}\weakto w^i$ in $L^2([0,T];L^2(\Omega)^3)$ as $k\to\infty$.  A reasoning similar to the above shows that $\p_{x_i}v=w^i$ and $v\in L^2([0,T];H^1_{0,\sigma}(\Omega))$.  

Since $\{\rho_{a_3(k)}v_{a_3(k)}\}_{k\in\N}$ is bounded in $L^2([0,T];L^2(\Omega)^3)$, there exists  a subsequence $\{ a_4(k) \}_{k\in\N}\subset \{a_3(k)\}_{k\in\N}$ and  $V\in L^2([0,T];L^2(\Omega)^3)$ such that $\rho_{a_4(k)}v_{a_4(k)}\weakto V$  in $L^2([0,T];L^2(\Omega)^3)$ as $k\to\infty$.

 We conclude that $\{\rho_{a_4(k)},v_{a_4(k)}\}_{k\in\N}$ is the desired subsequence. 
\end{proof}

We discuss strong convergence. Let $\{\rho_k,v_k\}_{k\in\N}$ be the subsequence mentioned in Proposition \ref{weak-convergence}. 
With $\{a_4(k)\}_{k\in\N}$ defined in the proof of Proposition \ref{weak-convergence}, $\{\tilde{\rho}_{a_4(k)}\}_{k\in\N}$ is also denoted by $\{\tilde{\rho}_k\}_{k\in\N}$.
The discrete parameter corresponding to $\rho_k, \tilde{\rho}_k, v_k$ is denoted by $\tau_k$. 
\begin{Prop}\label{P41}
 $\{\tilde{\rho}_k\}_{k\in\N}$ being seen as the sequence of $\tilde{\rho}_k:[0,T]\to L^2(\Omega)$, $k\in\N$ is weakly equi-continuous in the sense that for each $\phi\in C^\infty_0(\Omega)$, $\{(\tilde{\rho}_k,\phi)_{L^2(\Omega)}\}_{k\in\N}$ is equi-continuous on $[0,T]$. 
\end{Prop}
\begin{proof}
Let $\tilde{\rho}_k$ be generated by the solutions $\rho^n_k$ of \eqref{ellip1}$_{\tau=\tau_k}$. We fix an arbitrary $\phi\in C^\infty_0(\Omega)$. Observe that for $t\in[\tau_k n,\tau_k n+\tau_k]$, 
\begin{eqnarray*}
F_k(t)&:=&(\tilde{\rho}_k(t),\phi)_{L^2(\Omega)}\\
&=&(\rho^n_k,\phi)_{L^2(\Omega)}
-(v^n_{\tau_k}\cdot \nabla \rho^{n+1}_k,\phi)_{L^2(\Omega)}(t-\tau_k n)
+\theta(\Delta \rho^{n+1}_k,\phi)_{L^2(\Omega)}(t-\tau_k n)\\
&=& (\rho^n_k,\phi)_{L^2(\Omega)}+(v^n_{\tau_k}\rho^{n+1}_k,\nabla \phi)_{L^2(\Omega)^3}(t-\tau_k n)+\theta(\rho^{n+1}_k,\Delta\phi)_{L^2(\Omega)}(t-\tau_k n), 
\end{eqnarray*}
where there exists a constant $K(\phi)$ depending only on $\phi$ such that 
$$|(v^n_{\tau_k}\rho^{n+1}_k,\nabla\phi)_{L^2(\Omega)}|\le K(\phi),\quad \theta(\rho^{n+1}_k,\Delta\phi)_{L^2(\Omega)}\le K(\phi). $$
It is clear that, if $t,s\in[\tau_k n,\tau_k n+\tau_k]$, we have $|F_k(t)-F_k(s)|\le K(\phi)|t-s|$; hence, $F_k$ is $K(\phi)$-Lipschitz continuous on $[0,T]$ for all $k$.
\end{proof}
Due to Proposition \ref{P41} and Lemma \ref{key-lemma1}, we  find a subsequence $\{\tilde{\rho}_{a(k)}\}_{k\in\N}$ of $\{\tilde{\rho}_k\}_{k\in\N}$ that satisfies (A1)--(A3). Let $\{\tau_{a(k)}\}_{k\in\N}$, $\{\rho_{a(k)}\}_{k\in\N}$, $\{\tilde{\rho}_{a(k)}\}_{k\in\N}$, $\{\tilde{v}_{a(k)}\}_{k\in\N}$ be re-denoted by $\{\tau_k\}_{k\in\N}$, $\{\rho_k\}_{k\in\N}$, $\{\tilde{\rho}_k\}_{k\in\N}$, $\{v_k\}_{k\in\N}$, respectively. Then, the pair $\{\tilde{\rho}_k\}_{k\in\N}$, $\{v_k\}_{k\in\N}$ satisfies (A1)--(A4) and Lemma \ref{key-lemma2} implies:    
for each $\lambda>0$, there exists a constant $A_\lambda\ge0$ such that
\begin{eqnarray*}
&& \norm v_k(t,\cdot)-v_l(t,\cdot) \norm_{L^2(\Omega)^3)}\le \lambda(\norm v_k(t,\cdot)\norm_{H^1(\Omega)^3}+\norm v_l(t,\cdot)\norm_{H^1(\Omega)^3} +k^{-1}+l^{-1})\\\nonumber
&&\qquad  +A_\lambda\Big( \sup_{\varphi\in S}\Big|\Big( \tilde{\rho}_k(t,\cdot)v_k(t,\cdot)-\tilde{\rho}_l(t,\cdot)v_l(t,\cdot),\varphi \Big)_{L^2(\Omega)^3}\Big| +k^{-1}+l^{-1} \Big),\\
&&\mbox{for all $ t\in[0,T]$ and all $ k,l\in\N$}.
 \end{eqnarray*}
For each $\varphi\in S$, we have  
\begin{eqnarray*}
&&\Big|\Big( \tilde{\rho}_k(t,\cdot)v_k(t,\cdot)-\tilde{\rho}_l(t,\cdot)v_l(t,\cdot),\varphi \Big)_{L^2(\Omega)^3}-\Big( \rho_k(t,\cdot)v_k(t,\cdot)-\rho_l(t,\cdot)v_l(t,\cdot),\varphi \Big)_{L^2(\Omega)^3}\Big|\\
&&\le \norm \tilde{\rho}_k(t,\cdot)- \rho_k(t,\cdot)\norm_{L^2(\Omega)}\norm v_k(t,\cdot)\norm_{L^2(\Omega)^3}+\norm \tilde{\rho}_l(t,\cdot)- \rho_l(t,\cdot)\norm_{L^2(\Omega)}\norm v_l(t,\cdot)\norm_{L^2(\Omega)^3},\\
&&\sup_{\varphi\in S}\Big|\Big(\tilde{\rho}_k(t,\cdot)v_k(t,\cdot)-\tilde{\rho}_l(t,\cdot)v_l(t,\cdot),\varphi \Big)_{L^2(\Omega)^3}\Big|\\
&&\quad \le  \sup_{\varphi\in S}\Big|\Big(\rho_k(t,\cdot)v_k(t,\cdot)-\rho_l(t,\cdot)v_l(t,\cdot),\varphi \Big)_{L^2(\Omega)^3}\Big|\\
&&\qquad +C_4(T) ( \norm \tilde{\rho}_k(t,\cdot)- \rho_k(t,\cdot)\norm_{L^2(\Omega)}+ \norm \tilde{\rho}_l(t,\cdot)- \rho_l(t,\cdot)\norm_{L^2(\Omega)}),  
\end{eqnarray*} 
 and hence,
\begin{eqnarray*}
&&\Big[\int_0^T \Big\{\sup_{\varphi\in S}\Big|\Big(\tilde{\rho}_k(t,\cdot)v_k(t,\cdot)-\tilde{\rho}_l(t,\cdot)v_l(t,\cdot),\varphi \Big)_{L^2(\Omega)^3}\Big|\Big\}^2 dt\Big]^\2\\
&&\le \Big[\int_0^T\Big\{ \sup_{\varphi\in S}\Big|\Big(\rho_k(t,\cdot)v_k(t,\cdot)-\rho_l(t,\cdot)v_l(t,\cdot),\varphi \Big)_{L^2(\Omega)^3}\Big| \Big\}^2dt\Big]^\2\\
&&\qquad +  C_4(T) ( \norm \tilde{\rho}_k- \rho_k\norm_{L^2([0,T];L^2(\Omega))}+ \norm \tilde{\rho}_l- \rho_l\norm_{L^2([0,T];L^2(\Omega))})  \\
 &&\le \Big[\int_0^T\Big\{ \sup_{\varphi\in S}\Big|\Big(\rho_k(t,\cdot)v_k(t,\cdot)-\rho_l(t,\cdot)v_l(t,\cdot),\varphi \Big)_{L^2(\Omega)^3}\Big| \Big\}^2dt\Big]^\2
 +C_3(T)C_4(T)(\tau_k+\tau_l).
\end{eqnarray*} 
 Therefore, we obtain
\begin{eqnarray}\label{4key555}
&& \norm v_k-v_l \norm_{L^2([0,T];L^2(\Omega)^3)}\le \lambda(\norm v_k\norm_{L^2([0,T];H^1(\Omega)^3}+\norm v_l\norm_{L^2([0,T];H^1(\Omega)^3} )\\\nonumber
&&\qquad + \Big[\int_0^T\Big\{ \sup_{\varphi\in S}\Big|\Big(\rho_k(t,\cdot)v_k(t,\cdot)-\rho_l(t,\cdot)v_l(t,\cdot),\varphi \Big)_{L^2(\Omega)^3}\Big| \Big\}^2dt\Big]^\2
\\\nonumber
&&\qquad +(k^{-1}+l^{-1})(\lambda+A_\lambda)\sqrt{T}+C_3(T)C_4(T)(\tau_k+\tau_l).
 \end{eqnarray}
\begin{Prop}\label{strong-convergence1}
Let $v$ be the weak limit of $\{v_k\}_{k\in\N}$ mentioned in Proposition \ref{weak-convergence}. It holds that $\{v_k\}_{k\in\N}$ converges to $v$ strongly  in $L^2([0,T];L^2(\Omega)^3)$ as $k\to\infty$.  Furthermore,  $v\in L^\infty([0,T],L^2(\Omega)^3)$.  
\end{Prop}
\begin{proof}
Since $\{v_k\}_{k\in\N}$ is bounded in $L^\infty([0,T],L^2(\Omega)^3)$, its strong convergence implies that $v\in L^\infty([0,T],L^2(\Omega)^3)$ (consider an a.e. $t$-pointwise convergent subsequence). 
 
In \eqref{4key555}, we may choose $\lambda>0$ so that  $\lambda(\norm v_k\norm_{L^2([0,T];H^1(\Omega)^3)}+\norm v_l\norm_{L^2([0,T];H^1(\Omega)^3)})$ is arbitrarily small independently from $k,l\in\N$. Due to Lebesgue's dominated convergence theorem, the following pointwise convergence  
\begin{eqnarray}\label{454332}
\qquad  \sup_{\varphi\in S}\Big|\Big(\rho_k(t,\cdot)v_k(t,\cdot)-\rho_l(t,\cdot)v_l(t,\cdot),\varphi \Big)_{L^2(\Omega)^3}\Big| \to0 \mbox{ as $k,l\to\infty$,   $\forall\,t\in(0,T)$ }
 \end{eqnarray} 
implies that $\{v_k\}_{k\in\N}$ is a Cauchy sequence in $L^2([0,T];L^2(\Omega)^3)$ to conclude our assertion. 

We prove \eqref{454332} through the discrete time-derivative of $\rho_kv_k$ and $\rho_lv_l$.   Fix an arbitrary $t\in(0,T)$. Let $n_k\in \N$ be such that $t\in(\tau_k n_k,\tau_kn_k+\tau_k]$. For $\tilde{t}\in(t,T)$, let $\tilde{n}_k\in \N$ be such that $\tilde{t}\in(\tau_k \tilde{n},\tau_k\tilde{n}_k+\tau_k]$. We have 
$$0<\tau_{k}(\tilde{n}_{k}-n_{k})-\tau_{k}\le \tilde{t}-t\le \tau_{k}(\tilde{n}_{k}-n_{k})+\tau_{k}$$
 for all sufficiently large $k,l$. Later, $\tilde{t}$ is appropriately taken to be close enough to $t$.  
The time-discrete solutions that give the step functions $\rho_k,v_k$ are denoted by $\rho^n_k$, $v^n_k$, while $f^{n+1}_k$ denotes \eqref{333fff} with $\tau=\tau_k$.   Define   
\begin{eqnarray*}
a_{k}&:=&\frac{1}{\tau_{k}(\tilde{n}_{k}-n_{k})}\sum_{n=n_{k}+1}^{\tilde{n}_{k}} \rho_{k}^{n+1}v_{k}^{n+1} \tau_{k},\\
b_{k}&:=&\frac{1}{\tau_{k}(\tilde{n}_{k}-n_{k})}\!\!\sum_{n=n_{k}+1}^{\tilde{n}_{k}}\tau_{k}\{(n-1)-\tilde{n}_{k}\} \frac{\rho_{k}^{n+1}v_{k}^{n+1}-\rho_{k}^{n}v^{n}_{k} }{\tau_{k}} \tau_{k}\\
&=&\frac{1}{\tilde{n}_{k}-n_{k}}\sum_{n=n_{k}+1}^{\tilde{n}_{k}}\Big[ (n-\tilde{n}_{k}) \rho_{k}^{n+1}v_{k}^{n+1}- \{(n-1)-\tilde{n}_{k}\}\rho_{k}^{n}v^{n}_{k} \Big]
-a_{k},  
\end{eqnarray*}
which leads to 
$$\rho_{k}^{n_{k}+1}v_{k}^{n_{k}+1}=a_{k}+b_{k}.$$ 
We introduce $n_{l}$, $\tilde{n}_{l}$, $a_{l}$ and  $b_{l}$ in the same way with the same $t$ and $\tilde{t}$, to have $\rho_{l}^{n_{l}+1}v_{l}^{n_{l}+1}=a_{l}+b_{l}$. 
Fix an arbitrary $\varphi \in S$. We have 
\begin{eqnarray*}
&&\Big|\Big(\rho_k(t,\cdot)v_k(t,\cdot)-\rho_l(t,\cdot)v_l(t,\cdot),\varphi \Big)_{L^2(\Omega)^3}\Big| 
= \Big|\Big(\rho^{n_k+1}_kv^{n_k+1}_k-\rho^{n_l+1}_lv^{n_l+1}_l,\varphi \Big)_{L^2(\Omega)^3}\Big| \\
&&\le |(a_k-a_l,\varphi )_{L^2(\Omega)^3}|
+|(b_k,\varphi )_{L^2(\Omega)^3}|
+|(b_l,\varphi )_{L^2(\Omega)^3}|.
\end{eqnarray*}
We will show that $|(b_k,\varphi )_{L^2(\Omega)^3}|$ can be arbitrarily small as $\tilde{t}\to t$ independently from $k$ and  the choice of $\varphi\in S$. Hereafter,  $M_1,M_2,\ldots$ are some constants independent of $t$, $\tilde{t}$, $k$ and $\varphi\in S$.   Using \eqref{ellip2} in the form of \eqref{weak31}, we get  
\begin{eqnarray*}
&&|(b_k,\varphi )_{L^2(\Omega)^3}|
\le \sum_{n=n_{k}+1}^{\tilde{n}_{k}}\Big|\Big( \frac{\rho_{k}^{n+1}v_{k}^{n+1}-\rho_{k}^{n}v^{n}_{k} }{\tau_{k}} ,\varphi\Big)_{L^2(\Omega)^3}\Big|\tau_{k} \\
&&\le  \underline{ \sum_{n=n_{k}+1}^{\tilde{n}_{k}}
\sum_{j=1}^3\Big|\int_\Omega \rho^{n+1}_k v^n_{\tau_k j} v^{n+1}_{ki}\p_{x_j} \varphi_i dx\Big|\tau_{k} }_{\rm R_1}\\
&&\quad  +\underline{\sum_{n=n_{k}+1}^{\tilde{n}_{k}}\Big| \2\sum_{i,j=1}^3\int_\Omega \mu(\rho^{n+1}_k) (\p_{x_j} v^{n+1}_{ki}+\p_{x_i} v^{n+1}_{kj}  ) (\p_{x_j} \varphi_{i}+\p_{x_i} \varphi_{j} ) dx\Big|\tau_{k} }_{\rm R_2} \\
&&\quad +  \underline{\sum_{n=n_{k}+1}^{\tilde{n}_{k}}\Big| \theta\sum_{i,j=1}^3\int_\Omega \{ (\p_{x_i}\rho^{n+1}_k)v^{n+1}_{kj}(\p_{x_i}\varphi_j)+ (\p_{x_j}\rho^{n+1}_k)v^{n+1}_{ki}(\p_{x_i}\varphi_j)     
  \}dx\Big|\tau_{k}}_{\rm R_3}\\
&&\quad +  \underline{\sum_{n=n_{k}+1}^{\tilde{n}_{k}}\Big| 2\theta \sum_{i,j=1}^3\int_\Omega  \frac{\mu'(\rho^{n+1}_k)}{\rho^{n+1}_k} (\p_{x_i}\rho^{n+1}_k)(\p_{x_j}\rho^{n+1}_k)(\p_{x_j}\varphi_i)dx\Big|\tau_{k}}_{\rm R_4}\\
&&\quad +  \underline{\sum_{n=n_{k}+1}^{\tilde{n}_{k}}\Big|\theta^2\sum_{i,j=1}^3\int_\Omega\frac{1}{\rho^{n+1}_k}
 (\p_{x_i}\rho^{n+1}_k)(\p_{x_j}\rho^{n+1}_k)(\p_{x_i}\varphi_j)dx\Big|\tau_{k}}_{\rm R_5}\\
&&\quad  + \underline{ \sum_{n=n_{k}+1}^{\tilde{n}_{k}}\Big| \int_\Omega\rho^{n+1}_k f^{n+1}_k
\cdot\varphi dx\Big|\tau_{k} }_{\rm R_6}.
\end{eqnarray*}
By the results of Section 3, we see that 
\begin{eqnarray*}
&&\!\!\!R_1
\le M_1(\tilde{t}-t),\\
&&\!\!\!R_2
 \le    M_2\sum_{n=n_{k}+1}^{\tilde{n}_{k}} 
\sum_{i,j=1}^3 \norm \p_{x_j} v^{n+1}_{ki}\norm_{L^2(\Omega)}\tau_{k}
\le M_2 \sum_{i,j=1}^3 \norm \p_{x_j} v_{ki}\norm_{L^2([0,T];L^2(\Omega))}\sqrt{\tilde{t}-t}\\
&&\quad \le M_3\sqrt{\tilde{t}-t};
\end{eqnarray*}
in the same way, $R_3\le M_4 (\tilde{t}-t)$,  $R_4\le M_5 (\tilde{t}-t)$,  $R_5\le M_6 (\tilde{t}-t)$,  $R_6\le M_7 \sqrt{\tilde{t}-t}$. 
Hence, for any $\ep>0$, we may choose $\tilde{t}>t$ so that  $|(b_k,\varphi )_{L^2(\Omega)^3}|<\ep$ holds for all $\varphi\in S$ and $k\in\N$.   
The same reasoning yields $|(b_l,\varphi )_{L^2(\Omega)^3}|<\ep$. 

With this $\tilde{t}$ and the weak limit $V$ of $\{\rho_kv_k\}_{k\in\N}$ from Proposition \ref{weak-convergence}, we see that 
\begin{eqnarray*}
&&|(a_k-a_l,\varphi)_{L^2(\Omega)^3}|\\
&&\le\Big| \frac{1}{\tilde{t}-t}\int_{t}^{\tilde{t}} ( \rho_k(s,\cdot)v_k(s,\cdot),\varphi)_{L^2(\Omega)^3} -\frac{1}{\tilde{t}-t}\int_t^{\tilde{t}} (V(s,\cdot),\varphi)_{L^2(\Omega)^3} ds\Big|\\
&&\quad +\Big|\frac{1}{\tilde{t}-t}\int_t^{\tilde{t}} (V(s,\cdot),\varphi)_{L^2(\Omega)^3} ds
- \frac{1}{\tilde{t}-t}\int_{t}^{\tilde{t}} ( \rho_l(s,\cdot)v_l(s,\cdot),\varphi)_{L^2(\Omega)^3}
 \Big|+ M_8\frac{\tau_k+\tau_l}{\tilde{t}-t}\\
 &&\to0 \mbox{ as $l,k\to\infty$},
\end{eqnarray*}
where this convergence is uniform with respect to $\varphi\in S$. Thus, we conclude \eqref{454332} and complete the proof. 
\end{proof}
In order to take care of the nonlinearity of $\p_{x_i}\rho$ \eqref{NS1}, we need to  prove the strong convergence of  $\{\p_{x_i}\rho_k\}_{k\in\N}$ in $L^2([0,T];L^2(\Omega))$ for $i=1,2,3$. This issue is done with the interpolation inequality: {\it for each $\lambda>0$, there exists a constant $\tilde{A}_\lambda$ such that  
\begin{eqnarray}\label{4AL}
\norm g\norm_{L^2(\Omega)}\le\lambda \norm g\norm_{H^1(\Omega)}+\tilde{A}_\lambda\sup_{\phi\in\tilde{S}} |(g,\phi)_{L^2(\Omega)}|,\quad\forall\,g\in H^1(\Omega),
\end{eqnarray}
where $\tilde{S}:=\{\phi\in C^\infty_0(\Omega)\,|\,\norm \phi\norm_{W^{1,\infty}(\Omega)}=1\}$.} Note that \eqref{4AL} is an example of interpolation inequalities appearing in the Aubin-Lions lemma (see, e.g., Chapter 3.2 in \cite{Temam-book}), which is proven in a similar way to the proof of Lemma  \ref{key-lemma2} with the fact that $C^\infty_0(\Omega)$ is dense in $L^2(\Omega)$.
\begin{Prop}\label{strong-convergence2}
Let $\rho$ be the weak limit of $\{\rho_k\}_{k\in\N}$ mentioned in Proposition \ref{weak-convergence}. It holds that $\{\rho_k\}_{k\in\N}$, $\{\p_{x_i}\rho_k\}_{k\in\N}$  ($i=1,2,3$) converge to $\rho$, $\p_{x_i}\rho$, respectively, strongly  in $L^2([0,T];L^2(\Omega))$ as $k\to\infty$. Furthermore,  $\rho\in L^\infty([0,T],H^1(\Omega))$.     
\end{Prop}
\begin{proof} 
Since $\{\p_{x_i}\rho_k\}_{k\in\N}$ is bounded in $L^\infty([0,T],L^2(\Omega))$, its strong convergence implies that $\rho\in L^\infty([0,T],H^1(\Omega))$ (consider an a.e. $t$-pointwise convergent subsequence). 

Our proof is essentially the same as the proof of Proposition \ref{strong-convergence1}. 
Poincare's inequality gives 
$\norm \rho_k-\rho_l\norm_{L^2([0,T];L^2(\Omega))}\le A_P\norm \nabla \rho_k- \nabla\rho_l\norm_{L^2([0,T];L^2(\Omega))}+\sqrt{T}\tau_k+\sqrt{T}\tau_l$, 
where we note that  $\int_\Omega(\rho_k(t,\cdot)-\rho_l(t,\cdot))dx\equiv\int_\Omega(\eta_{\tau_k}(x)-\eta_{\tau_l}(x))dx$ and $\norm\eta_{\tau_k}-\eta_l\norm_{H^1(\Omega)}\le \tau_k+\tau_l$.    
Hence, it is enough to prove the strong convergence of $\{\p_{x_i}\rho_k\}_{k\in\N}$. We apply \eqref{4AL} to $\p_{x_i}\rho_k(t,\cdot)-\p_{x_i}\rho_l(t,\cdot)$ for each $t\in[0,T]$ and $k,l\in\N$ to obtain 
\begin{eqnarray}\label{24key555} 
&&\norm \p_{x_i}\rho_k-\p_{x_i}\rho_l\norm_{L^2([0,T];L^2(\Omega))}
\le \lambda \norm \p_{x_i}\rho_k-\p_{x_i}\rho_l\norm_{L^2([0,T];H^1(\Omega))}\\\nonumber
&&\quad +\tilde{A}_\lambda\Big[\int_0^T\Big\{\sup_{\phi\in\tilde{S}}\Big|\Big(\p_{x_i}\rho_k(t,\cdot)-\p_{x_i}\rho_l(t,\cdot),\phi\Big)_{L^2(\Omega)}\Big|\Big\}^2dt\Big]^\2.  
\end{eqnarray} 
Since $ \norm \p_{x_i}\rho_k-\p_{x_i}\rho_l\norm_{L^2([0,T];H^1(\Omega))}$ is bounded independently from $k,l$, it is enough to show the following pointwise convergence  in \eqref{24key555}   
\begin{eqnarray}\label{SSSS}
&& \sup_{\phi\in\tilde{S}}\Big|\Big(\p_{x_i}\rho_k(t,\cdot)-\p_{x_i}\rho_l(t,\cdot),\phi\Big)_{L^2(\Omega)}\Big| 
=\sup_{\phi\in\tilde{S}}\Big|\Big(\rho_k(t,\cdot)-\rho_l(t,\cdot),\p_{x_i}\phi\Big)_{L^2(\Omega)}\Big| \\\nonumber
&&\qquad \to0 \mbox{ as $k,l\to\infty$,   $\forall\,t\in(0,T)$ }.
 \end{eqnarray}  
Fix an arbitrary $t\in(0,T)$. Let $n_k\in \N$ be such that $t\in(\tau_k n_k,\tau_kn_k+\tau_k]$. For $\tilde{t}\in(t,T)$, let $\tilde{n}_k\in \N$ be such that $\tilde{t}\in(\tau_k \tilde{n},\tau_k\tilde{n}_k+\tau_k]$. We have 
$$0<\tau_{k}(\tilde{n}_{k}-n_{k})-\tau_{k}\le \tilde{t}-t\le \tau_{k}(\tilde{n}_{k}-n_{k})+\tau_{k}$$
 for all sufficiently large $k,l$. Later, $\tilde{t}$ is appropriately taken to be close enough to $t$.  
The time-discrete solutions that give the step function $\rho_k$ is denoted by $\rho^n_k$.   Define   
\begin{eqnarray*}
a_{k}&:=&\frac{1}{\tau_{k}(\tilde{n}_{k}-n_{k})}\sum_{n=n_{k}+1}^{\tilde{n}_{k}} \rho_{k}^{n+1} \tau_{k},\\
b_{k}&:=&\frac{1}{\tau_{k}(\tilde{n}_{k}-n_{k})}\!\!\sum_{n=n_{k}+1}^{\tilde{n}_{k}}\tau_{k}\{(n-1)-\tilde{n}_{k}\} \frac{\rho_{k}^{n+1}-\rho_{k}^{n} }{\tau_{k}} \tau_{k}.
\end{eqnarray*}
which leads to 
$$\rho_{k}^{n_{k}+1}=a_{k}+b_{k}.$$ 
We introduce $n_{l}$, $\tilde{n}_{l}$, $a_{l}$ and  $b_{l}$ in the same way with the same $t$ and $\tilde{t}$, to have $\rho_{l}^{n_{l}+1}=a_{l}+b_{l}$. 
Fix an arbitrary $\phi \in \tilde{S}$. We have 
\begin{eqnarray*}
&&\Big|\Big(\rho_k(t,\cdot)-\rho_l(t,\cdot),\p_{x_i}\phi \Big)_{L^2(\Omega)}\Big| 
= \Big|\Big(\rho^{n_k+1}_k-\rho^{n_l+1}_l,\p_{x_i}\phi \Big)_{L^2(\Omega)}\Big| \\
&&\le |(a_k-a_l,\p_{x_i}\phi )_{L^2(\Omega)}|
+|(b_k,\p_{x_i}\phi )_{L^2(\Omega)}|
+|(b_l,\p_{x_i}\phi )_{L^2(\Omega)}|.
\end{eqnarray*} 
We will see that $|(b_k,\p_{x_i}\phi )_{L^2(\Omega)}|$ can be arbitrarily small as $\tilde{t}\to t$ independently from $k$ and  the choice of $\phi\in \tilde{S}$. Hereafter,  $M'_1,M'_2,M'_3$ are some constants independent of $t$, $\tilde{t}$, $k$ and $\phi\in \tilde{S}$.   Using \eqref{ellip1}, we get  
\begin{eqnarray*}
&&|(b_k,\varphi )_{L^2(\Omega)}|
\le \sum_{n=n_{k}+1}^{\tilde{n}_{k}}\Big|\Big( \frac{\rho_{k}^{n+1}-\rho_{k}^{n} }{\tau_{k}} ,\p_{x_i}\phi\Big)_{L^2(\Omega)}\Big|\tau_{k} \\
&&\le  \underline{ \sum_{n=n_{k}+1}^{\tilde{n}_{k}}
\Big|\Big( v^n_{\tau_k}\cdot\nabla\rho^{n+1}_k,\p_{x_i}\phi\Big)_{L^2(\Omega)}\Big|\tau_{k} }_{\rm R_1}
 +\underline{\sum_{n=n_{k}+1}^{\tilde{n}_{k}}\Big| \theta\Big( \Delta\rho^{n+1}_k,\p_{x_i}\phi\Big)_{L^2(\Omega)}\Big|\tau_{k} }_{\rm R_2}.
\end{eqnarray*}
By the results of Section 3, we have $R_1\le M'_1(\tilde{t}-t)$ and $R_2\le M'_2\sqrt{\tilde{t}-t}$. 
Hence, for any $\ep>0$, we may choose $\tilde{t}>t$ so that  $|(b_k,\p_{x_i}\phi )_{L^2(\Omega)}|<\ep$,  $|(b_l,\p_{x_i}\phi)_{L^2(\Omega)}|<3$ holds for all $\phi\in \tilde{S}$ and $k,l\in\N$.   
With this $\tilde{t}$ and the weak limit $\rho$ of $\rho_k$ and Proposition \ref{weak-convergence}, we see that 
\begin{eqnarray*}
&&|(a_k-a_l,\p_{x_i}\phi)_{L^2(\Omega)}|=|(\p_{x_i}a_k-\p_{x_i}a_l,\phi)_{L^2(\Omega)}|\\
&&\le\Big| \frac{1}{\tilde{t}-t}\int_{t}^{\tilde{t}} (\p_{x_i} \rho_k(s,\cdot),\phi)_{L^2(\Omega)} -\frac{1}{\tilde{t}-t}\int_t^{\tilde{t}} (\p_{x_i}\rho(s,\cdot),\phi)_{L^2(\Omega)} ds\Big|\\
&&\quad +\Big|\frac{1}{\tilde{t}-t}\int_t^{\tilde{t}} (\p_{x_i}\rho(s,\cdot),\phi)_{L^2(\Omega)} ds
- \frac{1}{\tilde{t}-t}\int_{t}^{\tilde{t}} ( \p_{x_i}\rho_l(s,\cdot),\phi)_{L^2(\Omega)}
 \Big|+ M'_3\frac{\tau_k+\tau_l}{\tilde{t}-t}\\
 &&\to0 \mbox{ as $l,k\to\infty$},
\end{eqnarray*}
where this convergence is uniform with respect to $\phi\in \tilde{S}$. Thus, we conclude \eqref{SSSS} and complete the proof. 
\end{proof}
\setcounter{section}{5}
\setcounter{equation}{0}
\section{Proof of main result} 
We prove Theorem \ref{main1}. Let $T>0$ be an arbitrary number. We first show that the pair of the limits $\rho,v$ obtained in Section 5 is a weak $[0,T]$-solution, and then extend it to $[0,\infty)$.
For this purpose, we convert \eqref{ellip1} and \eqref{ellip2} into weak forms. 
Let $T_\tau$ be such that $T\in[\tau T_\tau-\tau,\tau T_\tau)$. Set $t_n=\tau n$. Hereafter,  $\tilde{M}_1,\tilde{M}_2$ are some constants independent of $\tau$. 
For each $\phi\in C^\infty([0,T]\times\Omega;\R)$ with supp($\phi)\subset[0,T)\times\Omega$, we have for all sufficiently small $\tau>0$, 
\begin{eqnarray*}
&&\sum_{n=0}^{T_\tau-1}\frac{1}{\tau}(\rho^{n+1}-\rho^n)\phi(t_{n},\cdot)\tau
=\sum_{n=0}^{T_\tau-1}\frac{1}{\tau}(\rho^{n+1}\phi(t_{n+1},\cdot)
-\rho^{n}\phi(t_{n},\cdot))\tau\\
&&-\sum_{n=0}^{T_\tau-1}\rho^{n+1}\frac{1}{\tau}(\phi(t_{n+1},\cdot)-\phi(t_{n},\cdot))\tau
=-\rho^0\phi(0,\cdot)-\sum_{n=0}^{T_\tau-1}\rho^{n+1}\p_t\phi(t_{n+1},\cdot)\tau+ O(\tau),
\end{eqnarray*}
where $\phi\equiv0$ near $T$ and $\norm O(\tau)\norm_{L^2(\Omega)}\le\tilde{M}_1\tau$.   Hence,  $\int_\Omega$\eqref{ellip1}$\times \phi(t_n,\cdot)$ leads to  
\begin{eqnarray}\label{5weak1}
&& -\sum_{n=0}^{T_\tau-1}\int_\Omega\rho^{n+1}\p_t\phi(t_{n+1},\cdot)dx\tau 
+\sum_{n=0}^{T_\tau-1}\int_\Omega(v^n_\tau\cdot\nabla\rho^{n+1} -\theta\Delta\rho^{n+1})\phi(t_{n},\cdot)dx\tau\\\nonumber
&&\quad -\int_\Omega\eta_\tau\phi(0,\cdot)dx+O(\tau)=0.
\end{eqnarray}
Similarly,  for each $\varphi\in C^\infty([0,T]\times\Omega;\R^3)$ with $\nabla\cdot\varphi=0$ and supp($\varphi)\subset[0,T)\times\Omega$ we get  the weak form of \eqref{ellip2} as 
\begin{eqnarray}\label{5weak2}
&& -\sum_{n=0}^{T_\tau-1}\int_\Omega \rho^{n+1}v^{n+1}\cdot\p_t\varphi(t_{n+1},\cdot)dx\tau 
-\sum_{n=0}^{T_\tau-1} \sum_{j=1}^3\int_\Omega
\rho^{n+1}v^n_{\tau j} v^{n+1}\cdot\p_{x_j}\varphi(t_n,\cdot)dx\tau\\\nonumber 
&&+\sum_{n=0}^{T_\tau-1} \2\sum_{i,j=1}^3\int_\Omega
\mu(\rho^{n+1})(\p_{x_j}v^{n+1}_i+\p_{x_i}v^{n+1}_j)
(\p_{x_j}\varphi_i(t_{n},\cdot)+\p_{x_i}\varphi_j(t_{n},\cdot))dx\tau\\\nonumber
&&+\sum_{n=0}^{T_\tau-1} \sum_{i,j=1}^3\int_\Omega\theta\{(\p_{x_i}\rho^{n+1})v^{n+1}_j(\p_{x_i}\varphi_j(t_{n},\cdot))+(\p_{x_j}\rho^{n+1})v^{n+1}_i(\p_{x_i}\varphi_j(t_n,\cdot))\}dx\tau\\\nonumber 
&&+\sum_{n=0}^{T_\tau-1} \sum_{i,j=1}^3\int_\Omega2\theta\frac{\mu'(\rho^{n+1})}{\rho^{n+1}}(\p_{x_i}\rho^{n+1})(\p_{x_j}\rho^{n+1})(\p_{x_j}\varphi_i(t_n,\cdot))dx\tau\\\nonumber
&&-\sum_{n=0}^{T_\tau-1} \sum_{i,j=1}^3\int_\Omega\theta^2 \frac{1}{\rho^{n+1}}(\p_{x_i}\rho^{n+1})(\p_{x_j}\rho^{n+1})(\p_{x_i}\varphi_j(t_n,\cdot))dx\tau \\\nonumber 
&&-\sum_{n=0}^{T_\tau-1} \int_\Omega \rho^{n+1}f^{n+1}\cdot \varphi(t_n,\cdot)dx\tau -\int_\Omega\eta_\tau u\cdot\varphi(0,\cdot)dx+O(\tau)=0.
\end{eqnarray}
\begin{proof}[{\bf   Proof of Theorem \ref{main1}.}] We first prove the existence of a weak $[0,T]$-solution. 
Let $\{\rho_k\}_{k\in\N}$, $\{v_k\}_{k\in\N}$, $\rho$, $v$ be the ones mentioned in Proposition \ref{strong-convergence1} and Proposition \ref{strong-convergence2}. 
We show that the pair $\rho$, $v$ satisfies \eqref{weak1} and \eqref{weak2}. 
 
Since $ v^0_\tau=0$, $v^0=u$, $\norm v^{n+1}_\tau-v^{n+1}\norm_{H^1(\Omega)^3}\le \tau$ and $\norm \eta_\tau-\eta\norm_{H^1(\Omega)}\le \tau$, we have 
\begin{eqnarray*}
\Big|\sum_{n=0}^{T_\tau-1}\int_\Omega (v^n_\tau-v^n)\cdot\nabla\rho^{n+1} \phi(t_{n},\cdot)dx\tau\Big|
+\Big| \int_\Omega(\eta_\tau-\eta)\phi(0,\cdot)dx \Big|
\le \tilde{M}_2\tau. 
\end{eqnarray*}
Hence, \eqref{5weak1} can be re-written with $\rho_k$ and $v_k$ as 
\begin{eqnarray}\label{54321}
&&\int_0^T\int_\Omega\{ \rho_k(t,x) \p_t\phi(t,x) +(v_k(t-\tau_k,x)\cdot\nabla\rho_k(t,x) -\theta\Delta\rho_k(t,x) )\phi(t,x) \}dxdt\\\nonumber
&&\quad -\int_\Omega\eta(x)\phi(0,x)dx+O(\tau_k)=0. 
\end{eqnarray}
The weak/strong convergence of $\rho_k,\nabla \rho_k, \Delta\rho_k,v_k, \nabla v_k$ in \eqref{54321} leads to \eqref{weak1} as $k\to\infty$, where we use $\norm v(\cdot-\tau_k,\cdot)-v\norm_{L^2({\rm supp}(\phi))^3}\to0$ as $k\to\infty$ to take care of the small-time-shift in $v_k(t-\tau_k,x)$. 
In a similar way,  \eqref{5weak2} can be re-written with $\rho_k$ and $v_k$ as 
\begin{eqnarray}\label{5weak3}
&& \int_0^T\int_\Omega\Big[ -\rho_k(t,x)v_k(t,x)\cdot\p_t\varphi(t,x)  
-\sum_{j=1}\rho_k(t,x)v_{k j}(t-\tau_k,x) v_k(t,x)\cdot\p_{x_j}\varphi(t,x)\\\nonumber 
&&+ \sum_{i,j=1}^3\2
\mu(\rho_k(t,x))(\p_{x_j}v_{ki}(t,x)+\p_{x_i}v_{kj}(t,x))
(\p_{x_j}\varphi_i(t,x)+\p_{x_i}\varphi_j(t,x))\\\nonumber
&&+ \sum_{i,j=1}^3\theta\Big(\p_{x_i}\rho_k(t,x)v_{kj}(t,x)\p_{x_i}\varphi_j(t,x)+\p_{x_j}\rho_k(t,x)v_{ki}(t,x)\p_{x_i}\varphi_j(t,x)\Big)\\\nonumber 
&&+\sum_{i,j=1}^32\theta\frac{\mu'(\rho_k(t,x))}{\rho_k(t,x)}\p_{x_i}\rho_k(t,x)\p_{x_j}\rho_k(t,x)\p_{x_j}\varphi_i(t,x)\\\nonumber
&&- \sum_{i,j=1}^3\theta^2 \frac{1}{\rho_k(t,x)}\p_{x_i}\rho_k(t,x)\p_{x_j}\rho_k(t,x)\p_{x_i}\varphi_j(t,x) \Big]dxdt\\\nonumber 
&&-\sum_{n=0}^{T_{\tau_k}-1} \int_\Omega \rho^{n+1}_kf^{n+1}_k\cdot \varphi(t_n,\cdot)dx\tau_k -\int_\Omega\eta(x) u(x)\cdot\varphi(0,x)dx+O(\tau_k)=0.
\end{eqnarray}
Observe that 
\begin{eqnarray*}
&&\sum_{n=0}^{T_{\tau_k}-1} \int_\Omega \rho^{n+1}_kf^{n+1}_k\cdot \varphi(t_n,\cdot)dx\tau_k
=\sum_{n=0}^{T_{\tau_k}-1} \int_\Omega \rho^{n+1}_k(x)\Big(\int_{\tau_k n}^{\tau_k n+\tau_k} f(t,x)dt \Big)\cdot \varphi(t_n,x)dx\\
&&=\sum_{n=0}^{T_{\tau_k}-1}\int_{\tau_k n}^{\tau_k n+\tau_k}\int_\Omega \rho_k(t,x)f(t,x)\cdot \varphi(t,x)dxdt +O(\tau_k)\\
&&\to \int_0^T\int_\Omega \rho f\cdot\varphi dxdt\mbox{\quad as $k\to\infty$.}
\end{eqnarray*}
Since $\mu$ is $C^1$-smooth and $\{\rho_k\}_{k\in\N}$  converges to $\rho$ strongly in $L^2([0,T];L^2(\Omega))$ as $k\to\infty$ with $m\le \rho_k\le M$, we have $\mu(\rho_k)\to\mu(\rho)$, $\mu'(\rho_k)/\rho_k\to\mu'(\rho)/\rho$  strongly in $L^2([0,T];L^2(\Omega))$ as $k\to\infty$, where the second convergence is verified through a subsequence of $\{\rho_k\}_{k\in\N}$ that converges to $\rho$ pointwise almost everywhere.  Therefore, we conclude that the weak/strong convergence of $\rho_k,\nabla \rho_k, \Delta\rho_k,v_k, \nabla v_k$ in \eqref{5weak3} leads to \eqref{weak2} as $k\to\infty$ to obtain a weak $[0,T]$-solution of \eqref{NS1}. 

Recall that the sequences $\{\rho_k\}_{k\in\N}$, $\{ v_k\}_{k\in\N}$ that yield a weak $[0,T]$-solution of \eqref{NS1} is defined in $[0,\infty)\times\Omega$, i.e.,   $\{\rho_k|_{[0,T]\times\Omega}\}_{k\in\N}$,  $\{ v_k|_{[0,T]\times\Omega}\}_{k\in\N}$  yield the weak $[0,T]$-solution $\rho=\rho^T, v=v^T$. We may follows the above argument to obtain  subsequences  $\{\rho_{a_2(k)}\}_{k\in\N}\subset \{\rho_k\}_{k\in\N}$, $\{ v_{a_2(k)}\}_{k\in\N}\subset \{ v_k\}_{k\in\N}$    
that  yield a weak $[0,2T]$-solution $\rho^{2T}, v^{2T}$  of \eqref{NS1}, where we note that 
$$\rho^{2T}|_{[0,T]\times\Omega}= \rho^T, \quad v^{2T}|_{[0,T]\times\Omega}=v^T.$$  
In this way, for each $l\in\N$, we find subsequences  $\{\rho_{a_{l+1}(k)}\}_{k\in\N}\subset \{\rho_{a_{l}(k)}\}_{k\in\N}$, $\{v_{a_{l+1}(k)}\}_{k\in\N}$$\subset \{ v_{a_l(k)}\}_{k\in\N}$   ($a_1(k):=k$)  
that  yield a weak $[0,(l+1)T]$-solution $\rho^{(l+1)T}, v^{(l+1)T}$  of \eqref{NS1}, where we note that 
$$\rho^{(l+1)T}|_{[0,\tilde{l}T]\times\Omega}= \rho^{\tilde{l}T}, \quad v^{(l+1)T}|_{[0,\tilde{l}T]\times\Omega}=v^{\tilde{l}T},\quad \forall\,\tilde{l}\le l+1.$$  
Hence, the functions $\rho^\ast:[0,\infty)\times\Omega\to\R$,  $\rho^\ast(t,\cdot):= \rho^{lT}(t,\cdot)$ if $t\le lT$ and $v^\ast:[0,\infty)\times\Omega\to\R^3$, $v^\ast(t,\cdot):= v^{lT}(t,\cdot)$ if $t\le lT$ are well-defined. 
We conclude that the pair $\rho^\ast,v^\ast$  is a global weak solution of \eqref{NS1}.
%
\end{proof}


\medskip\medskip\medskip

\noindent{\bf Acknowledgement.} 
This work was written during Kohei Soga's one-year research stay  in Fachbereich Mathematik, Technische Universit\"at Darmstadt, Germany, with the grant Fukuzawa Fund (Keio Gijuku Fukuzawa Memorial Fund for the Advancement of Education and Research). He would like to express special thanks to Professor Dieter Bothe for his kind hosting in TU-Darmstadt.  
This work is also supported by JSPS Grant-in-aid for Young Scientists \#18K13443 and JSPS Grants-in-Aid for Scientific Research (C) \#22K03391. 

\medskip\medskip\medskip

\noindent{\bf Data availability.} 
Data sharing not applicable to this article as no datasets were generated or analysed during the current study.
\medskip\medskip


\end{document}